\documentclass{amsart}
\usepackage[utf8]{inputenc}
\usepackage{setspace}
\usepackage{mathrsfs}
\usepackage{amssymb}
\usepackage{latexsym}
\usepackage{amsfonts}
\usepackage{amsmath}
\usepackage{eucal}
\usepackage{bm}
\usepackage{bbm}
\usepackage{graphicx}
\usepackage[english]{varioref}
\usepackage[nice]{nicefrac}
\usepackage[all]{xy}
\usepackage{amsthm}
\usepackage{amssymb,amsthm,upref,amscd}

\def\op{\operatorname}

\def\mmod{\kern-1pt\operatorname{-mod}}

\newtheorem{theorem}{Theorem}[section]
\newtheorem{lemma}[theorem]{Lemma}

\newtheorem{remark}[theorem]{Remark}

\theoremstyle{proposition}

\newtheorem{Cv}[theorem]{\bf Convention}
\newtheorem{Prop}[theorem]{\bf Proposition}
\newtheorem{Cor}[theorem]{\bf Corollary}

\numberwithin{equation}{section}

\begin{document}

\title[Abstract parabolic induction]{Abstract induced modules for reductive algebraic groups with Frobenius maps}

\author{Xiaoyu Chen}
\address{Department of Mathematics, Shanghai Normal University,
100 Guilin Road, Shanghai 200234, PR China.}
\email{gauss\_1024@126.com}

\author{Junbin Dong}
\address{Institute of Mathematical Sciences, ShanghaiTech University, 393 Middle Huaxia Road, Pudong, Shanghai 201210, PR China.}
\email{dongjunbin1990@126.com}


\date{November 11, 2020}


\begin{abstract}
Let ${\bf G}$ be a connected reductive algebraic group defined over a finite field $\mathbb{F}_q$ of $q$ elements, and ${\bf B}$ be a Borel subgroup of ${\bf G}$  defined over $\mathbb{F}_q$. Let $\Bbbk$ be a field and we assume that $\Bbbk=\bar{\mathbb{F}}_q $ when $\text{char}\ \Bbbk=\text{char} \ \mathbb{F}_q$.
 We show that the abstract induced module $\mathbb{M}(\theta)=\Bbbk{\bf G}\otimes_{\Bbbk{\bf B}}\theta$ (here $\Bbbk{\bf H}$ is the group algebra of ${\bf H}$ over the field $\Bbbk$ and $\theta$ is a character of ${\bf B}$ over $\Bbbk$) has a composition series (of finite length) if $\text{char}\ \Bbbk\ne \text{char} \ \mathbb{F}_q$. In the case $\Bbbk=\bar{\mathbb{F}}_q$ and $\theta$ is a rational character, we give a necessary and sufficient condition for the existence of a composition series (of finite length) of $\mathbb{M}(\theta)$. We determine all the composition factors whenever a composition series exists. Thus we obtain a large class of abstract infinite-dimensional irreducible $\Bbbk{\bf G}$-modules.
\end{abstract}

\maketitle

\section{Introduction}
The decomposition of certain induced modules is extremely important in the representation theory of algebraic groups and finite groups of Lie type. One very important class of finite-dimensional representations arises from considering the induced module from a character of a Borel subgroup. It is well known that all finite-dimensional rational irreducible modules arise by ``inducing" (see \cite{Jan1} for details) one-dimensional representations of the Borel subgoup (known as the costandard modules), and the decomposition problem of such modules is known as Lusztig's conjecture (cf.~\cite{L1} and \cite{L2}) which is true for large characteristic (cf.~\cite{AJ} and \cite{Fie}) but not always valid in  smaller characteristic (cf.~\cite{Wil}). Each irreducible representation of a finite group with split $BN$-pair of characteristic $p$ over a field of characteristic $p$ occurs in the head (socle) of some induced module from a character of $B$ (cf. \cite{CL}, \cite{SW}). In ordinary representation theory of finite reductive groups, each irreducible module occurs in some virtual representation (known as $R_T^\theta$, cf. \cite{DL}), and this classical theory is also contained in the textbook \cite{Car} and \cite{CE}. For finite groups with a certain set of subquotients, each irreducible module occurs at the head (socle) of the induced module from some ``cuspidal pair" (cf. \cite[Chapter 1]{CE}). There are analogous results in the representation theory of Lie algebras by considering the decomposition of Verma modules and baby Verma modules (cf.~\cite{Ver}, \cite{Ru}). Anyway, in representation theory it is often
a fundamental problem to determine the submodule structures of various classes of induced modules.

The induced modules from a one-dimensional module of a Borel subgroup of a finite reductive group have been investigated in great detail (cf. \cite{Jan2}, \cite{Pil}, and \cite{YY}).  For example, in \cite{Jan2} Jantzen constructed a filtration for such induced modules and gave a sum formula for these filtrations corresponding to  the well known Jantzen filtrations of Weyl modules. In \cite{Pil} C. Pillen proved that the socle and radical filtrations of such modules could be obtained from the filtrations of the generic Weyl modules under similar assumptions as in \cite{Jan2}. It was also showed in the same paper that these modules are rigid.

In contrast to the fruitful results discussed above, little is known for abstract induced modules for connected reductive algebraic groups with Frobenius maps (for example, $GL_n(\bar{\mathbb{F}}_q)$,  $SL_n(\bar{\mathbb{F}}_q)$, $SO_{2n}(\bar{\mathbb{F}}_q)$, $SO_{2n+1}(\bar{\mathbb{F}}_q)$, $Sp_{2n}(\bar{\mathbb{F}}_q)$,$\cdots$). Recently, Nanhua Xi studied certain infinite-dimensional representations of connected reductive groups over
a field of positive characteristic (cf. \cite{Xi}). These objects arise via induction from the group algebra of a Borel subgroup to the group algebra of the whole group. Xi constructed a submodule filtration of the abstract induced module from the trivial character of a Borel subgroup whose subquotients are indexed by the subsets of the set of simple reflections, and they turned out to be pairwise non-isomorphic.
Moreover, the authors of the present paper proved these subquotients are irreducible (see \cite{CD1} for the cross characteristic case and see  \cite{CD2}  for the case of arbitary base field). The first author also made an attempt to study the submodule structure of some induced modules in \cite{C1} and \cite{C2}.

In this paper, we study the existence of a composition series of the abstract induced module from an arbitary character of a Borel subgroup, and determine the composition factors whenever a composition series exists. Let ${\bf G}$ be a connected reductive algebraic group defined over a finite field $\mathbb{F}_q$ of $q$ elements, and $F$ be the standard Frobenius homomorphism on ${\bf G}$ induced by the automorphism $x\mapsto x^q$ on $\bar{\mathbb{F}}_q$. Let ${\bf B}$ be an $F$-stable Borel subgroup and ${\bf T}$ an $F$-stable maximal torus contained in ${\bf B}$. Let $\Bbbk$ be a field and we assume that $\Bbbk=\bar{\mathbb{F}}_q $ when $\text{char}\ \Bbbk=\text{char} \ \mathbb{F}_q$. This paper concerns the abstract induced modules $\mathbb{M}(\theta)=\Bbbk{\bf G}\otimes_{\Bbbk{\bf B}}\theta$ (here $\Bbbk{\bf H}$ is the group algebra of the group ${\bf H}$, and $\theta$ is a character of ${\bf T}$ regarded as a character of ${\bf B}$ by letting unipotent radical act trivially). We show that $\mathbb{M}(\theta)$ has a composition series if $\op{char}\Bbbk\neq\op{char}\bar{\mathbb{F}}_q$, in which case we determine all the composition factors of $\mathbb{M}(\theta)$. However, if $\Bbbk=\bar{\mathbb{F}}_q$, the situation is more complicated. Under the assumption that $\theta$ is a rational character of ${\bf T}$, we show that $\mathbb{M}(\theta)$ has a composition series if and only if $\theta$ is antidominant (see Section 4 for the definition), in which case the submodule structure is analogous to the  cross characteristic case.  In particular, we find a large class of infinite-dimensional irreducible abstract representations of ${\bf G}$.

Let us briefly give the idea of the proofs of our main results. The idea of the proof  of Theorem \ref{main} and \ref{antidominant} is similar to  the proof of \cite[Theorem 3.1]{CD1} and \cite[Theorem 4.1]{CD2} respectively.
However the arguments in this paper are more technical and we have to overcome some new challenges.
Let ${\bf U}$ be the unipotent radical of ${\bf B}$, and $U_{q^a}$ be its $\mathbb{F}_{q^a}$-points. In the case either $\op{char}\Bbbk\neq \op{char}\bar{\mathbb{F}}_q$ or $\Bbbk=\bar{\mathbb{F}}_q$ and $\theta$ is antidominant, we construct an explicit filtration of $\mathbb{M}(\theta)$. For any subquotient $E$ of this filtration, we prove the irreducibility of $E$ through the following steps: (1) Show that $E$ is a cyclic module (which is obvious by definition); (2) Show that any submodule of $E$ contains an $U_{q^a}$-fixed point for a sufficiently large integer $a$; (3) Show that any $U_{q^a}$-fixed point is transited by $\Bbbk{\bf G}$ to a generator of $E$. It is important to note that the step (3) is a new phenomenon in our case, and we develop a new and highly nontrivial technique to settle it. Some discussions  in the proof of  \cite[Theorem 3.1]{CD1} and \cite[Theorem 4.1]{CD2} also work here under some more general setting. For the simplicity of this article, we omit the proofs of these similar steps and spend more time explaining the new approach which  can deal with  a  general character  $\theta\in \widehat{\bf T}$ not just a trivial character which we consider in \cite{CD1}, \cite{CD2}.
In the case $\Bbbk=\bar{\mathbb{F}}_q$ and $\theta$ is not antidominant, by the transitivity of  Harish-Chandra induction (infinite version) and exactness of the (abstract) induction functor, the non-existence of finite composition series of $\mathbb{M}(\theta)$ reduces to the case ${\bf G}=SL_2(\bar{\mathbb{F}}_q)$. In this case, we prove the non-existence result using the limit process and the classical structural results of the Weyl modules for $SL_2(\bar{\mathbb{F}}_q)$ (cf.\cite{De}).

This paper is organized as follows: In Section 2 we recall some notations, and give some general constructions and results working for any field $\Bbbk$. In particular, it contains the general properties of the abstract induced modules $\mathbb{M}(\theta)$. Section 3 is devoted to studying the decomposition of $\mathbb{M}(\theta)$ in the case $\op{char}\Bbbk\neq \op{char}\bar{\mathbb{F}}_q$.
When $\Bbbk=\bar{\mathbb{F}}_q$ and $\theta$ is rational, Section 4 and Section 5 deal with the existence, respectively non-existence, of  a composition series of $\mathbb{M}(\theta)$ in the antidominant  and the non-antidominant case, respectively.
Using such decomposition of $\mathbb{M}(\theta)$ in previous sections, some more corollaries and conclusions are given in Section 6.

\section{General setting}

Let ${\bf G}$ be a connected reductive algebraic group defined over $\mathbb{F}_q$ with the standard Frobenius homomorphism $F$ induced by the automorphism $x\mapsto x^q$ on $\bar{\mathbb{F}}_q$. Let ${\bf B}$ be an $F$-stable Borel subgroup, and ${\bf T}$ be an $F$-stable maximal torus contained in ${\bf B}$, and ${\bf U}=R_u({\bf B})$ be the ($F$-stable) unipotent radical of ${\bf B}$.
We identify ${\bf G}$ with ${\bf G}(\bar{\mathbb{F}}_q)$ and do likewise for the various subgroups of ${\bf G}$ such as ${\bf B}, {\bf T}, {\bf U}$ $\cdots$. We denote by $\Phi=\Phi({\bf G};{\bf T})$ the corresponding root system, and by $\Phi^+$ (resp. $\Phi^-$) the set of positive (resp. negative) roots determined by ${\bf B}$. Let $W=N_{\bf G}({\bf T})/{\bf T}$ be the corresponding Weyl group. We denote by $\Delta=\{\alpha_i\mid i\in I\}$ the set of simple roots and by $S=\{s_i:=s_{\alpha_i}\mid i\in I\}$ the corresponding simple reflections in $W$. For each $\alpha\in\Phi$, let ${\bf U}_\alpha$ be the root subgroup corresponding to $\alpha$ and we fix an isomorphism $\varepsilon_\alpha: \bar{\mathbb{F}}_q\rightarrow{\bf U}_\alpha$ such that $t\varepsilon_\alpha(c)t^{-1}=\varepsilon_\alpha(\alpha(t)c)$ for any $t\in{\bf T}$ and $c\in\bar{\mathbb{F}}_q$. For any $w\in W$, let ${\bf U}_w$ (resp. ${\bf U}_w'$) be the subgroup of ${\bf U}$ generated by all ${\bf U}_\alpha$  with $w(\alpha)\in\Phi^-$ (resp. $w(\alpha)\in\Phi^+$). The multiplication map ${\bf U}_w\times{\bf U}_w'\rightarrow{\bf U}$ is a bijection (cf. \cite[Proposition 2.5.12]{Car}). For any $J\subset I$, let $W_J$ and ${\bf P}_J$ be the corresponding standard parabolic subgroup of $W$ and ${\bf G}$, respectively. One denotes by $w_J$ the longest element in $W_J$.
We also have the Levi decomposition ${\bf P}_J={\bf L}_J\ltimes{\bf U}_J$, where ${\bf L}_J$ is the subgroup of ${\bf P}_J$ generated by ${\bf T}$, and all ${\bf U}_{\alpha_i}$ and ${\bf U}_{-\alpha_i}$ with $i\in J$ and ${\bf U}_J=R_u({\bf P}_J)$.

Let $\Bbbk$ be a field and we assume that $\Bbbk=\bar{\mathbb{F}}_q $ when $\text{char}\ \Bbbk=\text{char} \ \mathbb{F}_q$. All representations of ${\bf G}$ we consider are over the field $\Bbbk$.
Denote by $\Bbbk{\bf G}$ the group algebra of ${\bf G}$. For a $F$-stable subgroup ${\bf H}$ of ${\bf G}$, we denote by $H_{q^a}$ the $\mathbb{F}_{q^a}$-points of ${\bf H}$ (equivalently, $F^a$-fixed points of ${\bf H}$). Thus, one identifies ${\bf H}$ with the union of all $H_{q^a}$. For any finite subset $X$ of ${\bf G}$, let $\underline{X}:=\sum_{x\in X}x \in\Bbbk{\bf G}$. This notation will be frequently used later. Without loss of generality, we make the following convention throughout this paper.

\begin{Cv}
We assume that all representatives of the elements of $W$ involved are in $G_q$ without loss of generality. $($Otherwise, we replace $q$ by a sufficiently large power of $q$. This does no harm to the result.$)$
\end{Cv}

Let $\widehat{\bf T}$ be the set of characters of ${\bf T}$ when $\text{char}\ \Bbbk\ne\text{char} \ \mathbb{F}_q$ and let it be the set of rational characters when  $\Bbbk=\bar{\mathbb{F}}_q $. Each $\theta\in\widehat{\bf T}$ can be regarded as a character of ${\bf B}$ by the homomorphism ${\bf B}\rightarrow{\bf T}$. Let ${\Bbbk}_\theta$ be the corresponding ${\bf B}$-module.
We are interested in the induced module $\mathbb{M}(\theta)=\Bbbk{\bf G}\otimes_{\Bbbk{\bf B}}{\Bbbk}_\theta$. Let ${\bf 1}_{\theta}$ be a fixed nonzero element in ${\Bbbk}_\theta$. We abbreviate $x{\bf 1}_{\theta}:=x\otimes{\bf 1}_{\theta}\in\mathbb{M}(\theta)$ for $x\in {\bf G}$.

\begin{Prop}\label{indec}
For any $\theta\in\widehat{\bf T}$, we have the isomorphism $\op{End}_{\Bbbk{\bf G}}(\mathbb{M}(\theta))\simeq\Bbbk$ as $\Bbbk$-algebras. In particular, the $\Bbbk {\bf G}$-module $\mathbb{M}(\theta)$ is indecomposable.
\end{Prop}

\begin{proof} Let $f\in\op{End}_{\Bbbk{\bf G}}(\mathbb{M}(\theta))$. Noting that $f({\bf 1}_\theta)\in \mathbb{M}(\theta)^{\bf U}$, it is enough to show  that $\mathbb{M}(\theta)^{\bf U}=\Bbbk{\bf 1}_\theta$. Using the Bruhat decomposition, we see that $$\mathbb{M}(\theta)= \sum_{w\in W}\Bbbk {\bf U} \dot{w} {\bf 1}_{\theta},$$
where $\dot{w}$ is a fixed representative of $w\in W$. Now let $\xi \in \mathbb{M}(\theta)^{\bf U}$ with the following expression
$$\xi= \sum_{w\in W}\sum_{x\in {\bf U}} a_{x,w} x \dot{w} {\bf 1}_\theta, \ \ \ \  a_{x,w}\in \Bbbk .$$
There exists a positive integer $m$ such that $x \in  U_{q^m}$ when $a_{x,w}\ne 0$.
Now assume that there is an element  $w\ne e$ ($e$ is the neutral element in $W$) such that $a_{x, w}\ne 0$ for some $x\in  U_{q^m}$. Then it is easy to see that  $u \xi \ne \xi$ for any $u \in {\bf U} \setminus U_{q^m} $ which is a contradiction. Thus $\xi =a {\bf 1}_\theta $ for some $a\in \Bbbk $ and we have $\mathbb{M}(\theta)^{\bf U}=\Bbbk{\bf 1}_\theta$ which completes the proof.
\end{proof}

\begin{remark}
{\rm Let $\op{tr}$ be the trivial $B_q$-module. It is well known that $\op{Ind}_{B_q}^{G_q}\op{tr}$ is always decomposable, and $\op{End}_{\Bbbk G_q}(\op{Ind}_{B_q}^{G_q}\op{tr})$ is known as Hecke algebra when $\Bbbk=\mathbb{C}$. Proposition \ref{indec} shows that this never happens for a ${\bf G}$-module which is obtained by abstract induction from a character of ${\bf B}$.}
\end{remark}

For each $i \in I$, let ${\bf G}_i$ be the subgroup of $\bf G$ generated by ${\bf U}_{\alpha_i}, {\bf U}_{-\alpha_i}$ and set ${\bf T}_i= {\bf T}\cap {\bf G}_i$. For each $i\in I$, there exists a surjective homomorphism  $\varphi_i: SL_2(\bar{\mathbb{F}}_q)\rightarrow{\bf G}_i$ such that
$$\varphi_i\left(\begin{array}{cc}1 &\ x\\0 &\ 1\end{array}\right)=\varepsilon_{\alpha_i}(x),~~\varphi_i\left(\begin{array}{cc}1 &\ 0\\x &\ 1\end{array}\right)=\varepsilon_{-\alpha_i}(x).$$
For $\theta\in\widehat{\bf T}$, define the subset $I(\theta)$ of $I$ by $$I(\theta)=\{i\in I \mid \theta| _{{\bf T}_i} \ \text {is trivial}\}.$$
The Weyl group $W$ acts naturally on $\widehat{\bf T}$ by
\begin{equation}\label{theta^w}
 (w\cdot \theta ) (t):=\theta^w(t)=\theta(\dot{w}^{-1}t\dot{w})
\end{equation}
 for any $\theta\in\widehat{\bf T}$. Denote by
$W_{\theta}$ the stabilizer of $\theta$. The following lemma is clear.

\begin{lemma}\label{Wtheta}
$\op{(i)}$ The parabolic subgroup $W_{I(\theta)}$ is a subgroup of $W_{\theta}$.

\noindent  $\op{(ii)}$ Conversely, if $W_{\theta}$ is a parabolic subgroup of $W$, then $W_\theta=W_{I(\theta)}$.
\end{lemma}

\begin{proof} $\op{(i)}$ Since $W_{I(\theta)}$ is generated by $s_i,  i\in I(\theta)$, it is sufficient to show $\theta^{s_i}=\theta $ for each $i\in I(\theta)$. We have
\begin{equation}\label{si theta}
\theta^{s_i}(t)= \theta(s_i^{-1}ts_i)=\theta(t) \theta(t^{-1}s_i^{-1}ts_i), \ \ \forall \ t\in {\bf T}.
\end{equation}
Now since $t^{-1}s_i^{-1}ts_i \in {\bf T}_i$, we get $\theta(t^{-1}s_i^{-1}ts_i)=1$. Then $\theta^{s_i}=\theta $ and $\op{(i)}$ is proved.

\smallskip
\noindent  $\op{(ii)}$ We  prove that $W_\theta\subset W_{I(\theta)}$ which implies that $W_\theta=W_{I(\theta)}$ by $\op{(i)}$.  Since by assumption $W_\theta$ is generated by some simple reflections, it is enough to show that if $\theta^{s_i}=\theta$, then $\theta|_{{\bf T}_i}$ is trivial. By (\ref{si theta}), it suffices to show that each element $g\in {\bf T}_i$ can be
written as $t^{-1}s_i^{-1}ts_i$ with $t \in {\bf T}_i$. It is enough to verify this in $SL_2(\bar{\mathbb{F}}_q)$ in which case $s_i=s= \begin{pmatrix}
0 & 1 \\ -1 & 0
\end{pmatrix}$ by the surjective homomorphism  $\varphi_i: SL_2(\bar{\mathbb{F}}_q)\rightarrow{\bf G}_i.$  For any fixed $g=\begin{pmatrix}
b & 0 \\ 0 & b^{-1}
\end{pmatrix}$,   let $x$ be a square root of $b^{-1}$ in $\bar{\mathbb{F}}_q$ and $t=\begin{pmatrix}
x & 0 \\ 0 & x^{-1}
\end{pmatrix}$. Then $t^{-1}s^{-1}ts=g$ which completes the proof.
\end{proof}

Let $J\subset I(\theta)$, and ${\bf G}_J$ be the subgroup of $\bf G$ generated by ${\bf G}_i$, $i\in J$. We choose a representative $\dot{w}\in {\bf G}_J$ for each $w\in W_J$. Thus, the element $w{\bf 1}_\theta:=\dot{w}{\bf 1}_\theta$  $(w\in W_J)$ is well-defined. For $J\subset I(\theta)$, we set
$$\eta(\theta)_J=\sum_{w\in W_J}(-1)^{\ell(w)}w{\bf 1}_{\theta},$$
and let $\mathbb{M}(\theta)_J=\Bbbk{\bf G}\eta(\theta)_J$.
Analogous to \cite[Proposition 2.3]{Xi}, we have the following proposition.

\begin{Prop}\label{MJ=KUW}
 For $J\subset I(\theta)$, the $\Bbbk {\bf G}$-module $\mathbb{M}(\theta)_J$ has the form $$\mathbb{M}(\theta)_J=\sum_{w\in W}\Bbbk{\bf U}\dot{w}\eta(\theta)_J=\sum_{w\in W}\Bbbk{\bf U}_{w_Jw^{-1}}\dot{w}\eta(\theta)_J.$$
\end{Prop}
\begin{proof}
The second equality follows immediately from the following claim which is also frequently used later.
We claim that  if $w\leq s_{\alpha}w$,  $u\dot{w}{\bf 1}_{\theta}=\dot{w}{\bf 1}_{\theta}$ for any $u\in {\bf U}_{\alpha}$. Indeed, $w\le s_\alpha w$ implies that $w^{-1}(\alpha)>0$, and hence $\dot{w}^{-1}{\bf U}_\alpha\dot{w}\subset{\bf U}$. Therefore, $u\dot{w}{\bf 1}_\theta=\dot{w}(\dot{w}^{-1}u\dot{w}){\bf 1}_\theta=\dot{w}{\bf 1}_\theta$. Now we let $$M=\displaystyle \sum_{w\in W}\Bbbk{\bf U}\dot{w}\eta(\theta)_J.$$ Since $M$ contains $\eta(\theta)_J$, to show the first equality it is enough to show that $M$ is a $\Bbbk{\bf G}$-submodule, and hence to show that $M$ is $N_{\bf G}({\bf T})$-stable since ${\bf G}$ is generated by ${\bf B}$ and $N_{\bf G}({\bf T})$. We have to check $\dot{s_i}u\dot{h} \eta(\theta)_J \in  M$ for any $u\in {\bf U}$, $h\in W$, and $i\in I$. Since each element $u\in {\bf U}$ can be written as $u=u_i'u_i$ with $u_i'\in {\bf U}'_{s_i}$ and $u_i\in {\bf U}_{s_i}$, noting that $\dot{s_i}u'_i\dot{s_i}^{-1}\in {\bf U}$, it is enough to check $\dot{s_i} u_i \dot{h} \eta(\theta)_J \in M$ . There is no harm to assume that $\ell(hw_J)=\ell(h)+\ell(w_J)$.

The case $u_i=1$ is clear. For each $u_i\in {\bf U}_{\alpha_i}\backslash\{1\}$, we have
\begin{equation}\label{sus}
\dot{s_i}u_i\dot{s_i}^{-1}=f_i(u_i)\dot{s_i}h_i(u_i)g_i(u_i),
\end{equation}
where $f_i(u_i),g_i(u_i) \in {\bf U}_{\alpha_i}\backslash\{1\}$, and $h_i(u_i)\in {\bf T}_i$ are uniquely determined. Then we need to deal with the following three cases.

\smallskip
\noindent $\op{(i)}$ If $hw_J\leq s_ihw_J$, then $hw\leq s_ihw$ for each $w\in W_J$. For any $w\in W$, we have
$$\dot{s_i} u_i \dot{h} \dot{w} {\bf 1}_{\theta}= \dot{s_i} \dot{h} \dot{w} (\dot{h} \dot{w})^{-1} u_i \dot{h} \dot{w} {\bf 1}_{\theta}. $$
When $w\in W_J$ and $hw\leq s_ihw$,  we get $(\dot{h} \dot{w})^{-1} u_i \dot{h} \dot{w} \in {\bf U} $.
Therefore in this case $$\dot{s_i} u_i \dot{h} \dot{w} {\bf 1}_{\theta}= \dot{s_i} \dot{h} \dot{w} {\bf 1}_{\theta} $$
which implies that
$$\dot{s_i} u_i \dot{h} \eta(\theta)_J =  \sum_{w\in W_J} (-1)^{\ell(w)} \dot{s_i} u_i \dot{h}  \dot{w} {\bf 1}_{\theta} =\dot{s_i}\dot{h} \eta(\theta)_J \in M.$$

\smallskip
\noindent  $\op{(ii)}$ If $s_ih \leq  h$, using  (\ref{sus}) we see that
$$\dot{s_i} u_i \dot{h} \eta(\theta)_J  =f_i(u_i)\dot{s_i}h_i(u_i)g_i(u_i) \dot{s_i} \dot{h} \eta(\theta)_J.$$ By the same discussion as $\op{(i)}$ we have
$g_i(u_i) \dot{s_i} \dot{h} \eta(\theta)_J= \dot{s_i} \dot{h} \eta(\theta)_J.$
Therefore it is not difficult to get
$$\dot{s_i} u_i \dot{h} \eta(\theta)_J  =f_i(u_i)\dot{s_i}h_i(u_i) \dot{s_i} \dot{h} \eta(\theta)_J=\theta^{s_ih}(h_i(u_i)) f_i(u_i)\dot{s_i}^2 \dot{h}\eta(\theta)_J .$$
Noting that $\dot{s_i}^2 \in {\bf T}$, the above equation becomes
\begin{equation}\label{pf2.5-1}
\aligned
\dot{s_i} u_i \dot{h} \eta(\theta)_J
&\ =\theta^{s_ih}(h_i(u_i))\theta^h(\dot{s_i}^2)f_i(u_i)\dot{h}\eta(\theta)_J\\
&\ =\theta^h(\dot{s_i}h_i(u_i)\dot{s_i})f_i(u_i)\dot{h}\eta(\theta)_J \in M
\endaligned
\end{equation}

\smallskip

\noindent  $\op{(iii)}$ If $h \leq  s_ih$ but $s_ihw_J\leq hw_J$, then $s_ih=hs_j$ for some $j\in J$. We can assume that $\dot{s_i}\dot{h}=\dot{h}\dot{s_j}t$, for some $t\in {\bf T}$ and $\dot{s_j}\in{\bf G}_j$. Therefore
$$
\aligned
\dot{s_i} u_i \dot{h} \eta(\theta)_J &\ =\dot{s_i}\dot{h}{\dot{h}}^{-1}u_i\dot{h}\eta(\theta)_J=\dot{h}\dot{s_j}t\dot{h}^{-1}u_i\dot{h}\eta(\theta)_J.
\endaligned
$$
For convenience we set $u_j=t\dot{h}^{-1}u_i\dot{h}t^{-1}\in {\bf U}_{\alpha_j}$, then the above equation becomes
\begin{equation}\label{siuih}
\dot{s_i} u_i \dot{h} \eta(\theta)_J= \theta(t)\dot{h}\dot{s_j}u_j\eta(\theta)_J.
\end{equation}

Firstly, it is clear that
$$\dot{s_j}u_j\eta(\theta)_J  =\dot{s_j}u_j\sum_{{w\in W_J}\atop{w\leq s_jw}} (-1)^{\ell(w)}w{\bf 1}_{\theta}+\dot{s_j}u_j\sum_{{w\in W_J}\atop{s_jw\leq w}} (-1)^{\ell(w)}w{\bf 1}_{\theta}$$
For the first part, we see that
$$
\aligned
\dot{s_j}u_j\sum_{{w\in W_J}\atop{w\leq s_jw}} (-1)^{\ell(w)}w{\bf 1}_{\theta}=\dot{s_j}\sum_{{w\in W_J}\atop{w\leq s_jw}} (-1)^{\ell(w)}w{\bf 1}_{\theta}= -\sum_{{w\in W_J}\atop{s_jw\leq w}} (-1)^{\ell(w)}w{\bf 1}_{\theta}.
\endaligned $$
For the second part, by (\ref{sus}) we have
$$
\aligned
\dot{s_j}u_j\sum_{{w\in W_J}\atop{s_jw\leq w}} (-1)^{\ell(w)}w{\bf 1}_{\theta} &\ = f_j(u_j)\dot{s_j}h_j(u_j)g_j(u_j)\dot{s_j}\sum_{{w\in W_J}\atop{s_jw\leq w}} (-1)^{\ell(w)}w{\bf 1}_{\theta}
\endaligned $$
which also equals to $\displaystyle f_j(u_j)\sum_{{w\in W_J}\atop{s_jw\leq w}} (-1)^{\ell(w)}w{\bf 1}_{\theta}$ by some easy computation.
Combining  these two equations  we get
$$\dot{s_j}u_j\eta(\theta)_J =-\sum_{{w\in W_J}\atop{s_jw\leq w}} (-1)^{\ell(w)}w{\bf 1}_{\theta}+ f_j(u_j)\sum_{{w\in W_J}\atop{s_jw\leq w}} (-1)^{\ell(w)}w{\bf 1}_{\theta}.$$
It can be also written as
$$ -\sum_{{w\in W_J}\atop{s_jw\leq w}} (-1)^{\ell(w)}w{\bf 1}_{\theta}+f_j(u_j)\eta(\theta)_J-f_j(u_j)\sum_{{w\in W_J}\atop{w\leq s_jw}} (-1)^{\ell(w)}w{\bf 1}_{\theta}
$$
Noting that
$$f_j(u_j)\sum_{{w\in W_J}\atop{w\leq s_jw}} (-1)^{\ell(w)}w{\bf 1}_{\theta} = \sum_{{w\in W_J}\atop{w\leq s_jw}} (-1)^{\ell(w)}w{\bf 1}_{\theta},$$
then we have $\dot{s_j}u_j\eta(\theta)_J =(f_j(u_j)-1) \eta(\theta)_J. $

\smallskip
By the above discussion, now (\ref{siuih}) becomes
$$
\aligned \dot{s_i} u_i \dot{h} \eta(\theta)_J = \theta(t)\dot{h}(f_j(u_j)-1) \eta(\theta)_J
 = \theta(t)(\dot{h}f_j(u_j)\dot{h}^{-1}-1)\dot{h}\eta(\theta)_J.
\endaligned $$
Noting that $u_j=t\dot{h}^{-1}u_i\dot{h}t^{-1}$,  combining (\ref{sus}) and the following two equations
$$\dot{h}\dot{s_j}u_j\dot{s_j}^{-1}\dot{h}^{-1}=\dot{h}\dot{s_j}t\dot{h}^{-1}u_i\dot{h}t^{-1}\dot{s_j}^{-1}
\dot{h}^{-1}=\dot{s_i}u_i\dot{s_i}^{-1},$$ $$\dot{h}\dot{s_j}u_j\dot{s_j}^{-1}\dot{h}^{-1}=(\dot{h}f_j(u_j)\dot{h}^{-1})(\dot{h}\dot{s_j}\dot{h}^{-1})
(\dot{h}h_j(u_j)\dot{h}^{-1})(\dot{h}g_j(u_j)\dot{h}^{-1}),$$
we get $f_i(u_i)=\dot{h}f_j(u_j)\dot{h}^{-1}$, and hence we have
\begin{equation}\label{pf2.5-2}
\dot{s_i} u_i \dot{h} \eta(\theta)_J =\theta(t)(f_i(u_i)-1)\dot{h}\eta(\theta)_J \in M.
\end{equation}

\smallskip
Now combining $\op{(i)}$, $\op{(ii)}$  and $\op{(iii)}$, the proposition is proved.
\end{proof}

For $w\in W$, denote by  $\mathscr{R}(w)=\{i\in I\mid ws_i<w\}$. For any subset $J\subset I$ and $K\subset I(\theta)$ we set
$$
\aligned
X_J &\ =\{x\in W\mid x~\op{has~minimal~length~in}~xW_J\};\\
Z_K &\ =\{w\in X_K \mid \mathscr{R}(ww_K)\subset K\cup (I\backslash I(\theta))\}.
\endaligned
$$
For each $w\in W$, let
$$C_w=\sum_{y\leq w}(-1)^{\ell(w)-\ell(y)}P_{y,w}(1)y\in\Bbbk W,$$
where $P_{y,w} \in \mathbb{Z}[t]$ is  the  Kazhdan-Lusztig polynomial associated to the pair $y, w\in W$ (cf. \cite[Theorem 1.1]{KL}). According to \cite[Theorem 1.1]{KL}, the elements $C_w$ with $w\in W$ form a basis of $\Bbbk W$. In particular, we have $C_{w_J}=\sum_{y\in W_J}(-1)^{\ell(w_Jy)}y$ because $P_{y,w_J}=1$ (cf. \cite[Lemma 2.6 (vi)]{KL}).  By \cite[Lemma 2.8 (c)]{Ge}, for $x\in X_J$, we have the uni-triangular relation
\begin{equation}\label{equ1}C_{xw_J}=xC_{w_J}+ \sum_{ y\in X_J, y<x}a_yyC_{w_J},\quad a_y\in \Bbbk\end{equation}
with its inverse
\begin{equation}\label{equ2} xC_{w_J}=C_{xw_J}+ \sum_{ y\in X_J, y<x}a'_yC_{yw_J},\quad a'_y\in\Bbbk.\end{equation}

\begin{lemma}\label{KL-basis}
Let $J\subset I(\theta)$.  Then the set $\{wC_{w_K}|J\subset K \subset I(\theta),~w\in Z_K\}$ forms a basis of  $\Bbbk WC_{w_J}$. In particular, we have
$$\Bbbk WC_{w_J}=\sum_{w\in Z_J}\Bbbk wC_{w_J} + \sum_{J\subsetneq K\subset I(\theta)}\Bbbk WC_{w_K}.$$
\end{lemma}
\begin{proof}
Firstly we note that $$X_Jw_J=\displaystyle \bigcup_{J\subset K \subset I(\theta)} Z_Kw_K \ \ \text{(disjoint union)}.$$ Let $V$ be the space spanned by $\{wC_{w_K}|J\subset K \subset I(\theta),~w\in Z_K\}$. It is enough to prove that $xC_{w_J}\in V$ for any $x\in X_J$. We show this by induction on $\ell(x)$.   The case $\ell(x)=0$ is trivial. Assume that $\ell(x)>0$. The result is trivial if $x\in Z_J$.  For $x\not\in Z_J$, we have $xw_J=yw_L$ for some $L\varsupsetneq J$ and $y\in Z_L$, and hence $C_{xw_J}=C_{yw_L}\in V$ by (\ref{equ1}) and induction. Moreover, we have
\begin{equation}\label{equ3}
xC_{w_J}=C_{xw_J}+\sum_{z<x,z\in X_J}k_zC_{zw_J},~~k_z\in\Bbbk
\end{equation}
by (\ref{equ2}).
It follows that $xC_{w_J}\in V$ by (\ref{equ1}), (\ref{equ3}), and induction. This completes the proof.
\end{proof}

For any $w\in W_{I(\theta)}$,  set $c_w=(-1)^{\ell(w)}C_w{\bf 1}_{\theta}\in \mathbb{M}(\theta).$
By \cite[Lemma 2.6 (vi)]{KL}, we  have $c_{w_J} = \eta(\theta)_J$ for any subset $J$ of $I(\theta)$. Since $\dot{s_i}c_{w_J}=-c_{w_J}$ if $i\in J$, we have
\begin{equation}\label{m=kuw}
\mathbb{M}(\theta)_J=\sum_{w\in X_J}\Bbbk {\bf U}_{w_Jw^{-1}}\dot{w}\eta(\theta)_J
\end{equation}
for any $J\subset I(\theta)$ by Proposition \ref{MJ=KUW}. Set
$$E(\theta)_J=\mathbb{M}(\theta)_J/\mathbb{M}(\theta)_J',$$
where $\mathbb{M}(\theta)_J'$ is the sum of all $\mathbb{M}(\theta)_K$ with $J\subsetneq K\subset I(\theta)$. This construction generalizes \cite[2.6]{Xi}. For each $J\subset I(\theta)$, denote by $C(\theta)_J$ the image of $\eta(\theta)_J$ in $E(\theta)_J$. The following proposition gives a basis of $E(\theta)_J$.

\begin{Prop}\label{DesEJ}
 The set $\{u\dot{w}C(\theta)_J \mid w\in Z_J, u\in {\bf U}_{w_Jw^{-1}} \}$ forms a basis of $E(\theta)_J$.
\end{Prop}
\begin{proof}
The set  is linearly independent by Lemma \ref{KL-basis}. By (\ref{m=kuw}) and Lemma \ref{KL-basis}, we see that $$E(\theta)_J=\sum_{w\in Z_J}\Bbbk {\bf U}_{w_Jw^{-1}}\dot{w}C(\theta)_J.$$ This completes the proof.
\end{proof}

\noindent The following proposition is analogous to \cite[Proposition 2.7]{Xi}.
\begin{Prop}\label{EJ}
Let $\theta_1,\theta_2\in\widehat{\bf T}$ and $K\subset I(\theta_1),L\subset I(\theta_2)$. Then $E(\theta_1)_K$ is isomorphic to $E(\theta_2)_L$ as $\Bbbk {\bf G}$-modules if and only if $K=L$ and $\theta_1=\theta_2$.
\end{Prop}
\begin{proof}
The ``if" part is clear. We prove the ``only if" part. Assume that there is an isomorphism $\phi: E(\theta_1)_K\rightarrow E(\theta_2)_L$. It is clear that $$t\phi(C(\theta_1)_K)=\theta_1(t)\phi(C(\theta_1)_K)$$ for any $t\in{\bf T}$. On the other hand, all ${\bf T}$-invariant lines in $E(\theta_2)_L$ are contained in $\Bbbk WC(\theta_2)_L$. Thus we get
\begin{equation}\label{isom}
\phi(C(\theta_1)_K)= \sum_{w\in Z_L} \lambda_w \dot{w}C(\theta_2)_L .
\end{equation}
 For any fixed $\theta\in \widehat{\bf T}$, $J\subset I(\theta)$ and $w\in Z_J$, we see that ${\bf U}_{\alpha}\dot{w} C(\theta)_J=\dot{w} C(\theta)_J$ if and only if ${\bf U}_{\alpha}  \nsubseteq  {\bf U}_{w_Jw^{-1}}$.
Therefore in (\ref{isom}), there is a unique $w\in Z_L$ such that $\lambda_w \ne 0$ and $w_K=ww_L$. However since $\dot{s_i}C(\theta_1)_K=-C(\theta_1)_K$ if and only if $i\in K$, $L$ has to be $K$ and then
$\phi(C(\theta_1)_K)= \lambda C(\theta_2)_K $  for some $\lambda \in \Bbbk$. Hence
$$ \theta_1(t)\phi(C(\theta_1)_K) =\phi(tC(\theta_1)_K)=t \phi(C(\theta_1)_K) = \theta_2(t)\phi(C(\theta_1)_K) $$
and we get $\theta_1= \theta_2$. This completes the proof.
\end{proof}

Analogous to \cite[Section 3]{CD1}, we have the following interpretation of $E(\theta)_J$ in terms of parabolic induction. Let $\theta\in\widehat{\bf T}$ and $K\subset I(\theta)$. Since $\theta|_{{\bf T}_i}$ is trivial for all $i\in K$, it induces a character (still denoted by $\theta$) of $\overline{\bf T}={\bf T}/{\bf T}\cap[{\bf L}_K,{\bf L}_K]$. Therefore, $\theta$ is regarded as a character of ${\bf L}_K$ by the homomorphism ${\bf L}_K\rightarrow\overline{\bf T}$ (with the kernel $[{\bf L}_K,{\bf L}_K]$), and hence as a character of ${\bf P}_K$ by letting ${\bf U}_K$ acts trivially. Set $\mathbb{M}(\theta, K):=\Bbbk{\bf G}\otimes_{\Bbbk{\bf P}_K}\theta$.  Let ${\bf 1}_{\theta, K}$ be a fixed nonzero element in the one-dimensional module $\Bbbk_\theta$ associated to $\theta$. We abbreviate $x{\bf 1}_{\theta, K}:=x\otimes{\bf 1}_{\theta, K} \in \mathbb{M}(\theta, K)$ as before. Let $J\subset I(\theta)$ and $J'=I(\theta)\backslash J$. Let $E(\theta)_J'$ be the submodule of $\mathbb{M}(\theta, J')$ generated by $D(\theta)_J:=\sum_{w\in W_J}(-1)^{\ell(w)}\dot{w}{\bf 1}_{\theta, J'}$.

\begin{Prop}\label{basis}
The set $\{u\dot{w}D(\theta)_J\mid w\in Z_J, u\in{\bf U}_{w_Jw^{-1}}\}$ forms a basis of $E(\theta)_J'$. In particular, $E(\theta)_J'$ is isomorphic to $E(\theta)_J$ as $\Bbbk {\bf G}$-modules.
\end{Prop}
\begin{proof}
Using Lemma \ref{MJ=KUW}, we have
$$E(\theta)_J'=\Bbbk {\bf G}D(\theta)_J = \sum_{w\in W}\Bbbk {\bf U} \dot{w} D(\theta)_J.$$
Since $\dot{s_i}D(\theta)_J=-D(\theta)_J$ for any $i\in J$, then $E(\theta)_J'=\sum_{w\in X_J}\Bbbk {\bf U} \dot{w} D(\theta)_J.$
Since $D(\theta)_K=(-1)^{\ell(w_K)}C_{w_K} {\bf 1}_{\theta, J'}$ for each $J \subset K\subset I(\theta)$, and $C_{w_K} {\bf 1}_{\theta, J'}=0$ for $J\subsetneq K\subset I(\theta)$,
we have $E(\theta)_J'=\sum_{w\in Z_J}\Bbbk {\bf U} \dot{w} D(\theta)_J$ by Lemma \ref{KL-basis}.
By the same argument as Proposition \ref{DesEJ}, we have $$E(\theta)_J'=\sum_{w\in Z_J}\Bbbk {\bf U}_{w_Jw^{-1}} \dot{w} D(\theta)_J,$$ and hence $E(\theta)_J'$ is isomorphic to $E(\theta)_J$ by Lemma \ref{DesEJ}.
\end{proof}

\section{The cross characteristic}

Throughout this section, we assume that $\op{char}\Bbbk\neq\op{char}\bar{\mathbb{F}}_q=p$. The main result of this section is the following
\begin{theorem}\label{main}
Let $\theta\in\widehat{\bf T}$. Then all $\Bbbk {\bf G}$-modules $E(\theta)_J$ $(J\subset I(\theta))$ are irreducible and pairwise non-isomorphic. In particular, the $\Bbbk {\bf G}$-module $\mathbb{M}(\theta)$ has exactly $2^{|I(\theta)|}$ composition factors, each occurring with  multiplicity one.
\end{theorem}

For each $i\in I$, we fix an $u_i\in {\bf U}_{\alpha_i}\backslash\{1\}$. Similar to \cite[Section 3]{CD1}, define
$$\tau_i:={u_i}^{-1}\dot{s_i}^{-1}(f_i(u_i)-1)\in\Bbbk{\bf G},$$ where $f_i(u_i)$ is defined by the formula (\ref{sus}). Therefore the operators $\tau_i$ $(i\in I)$ also have the following properties (Lemma \ref{tau}, Corollary \ref{tau2}, Corollary \ref{tau3} below), the proofs of which are identical to \cite[Lemma 3.3]{CD1}, \cite[Corollary 3.5]{CD1}, \cite[Corollary 3.6]{CD1}, respectively, as long as we replace $D_J$ there with $D(\theta)_J$.
\begin{lemma}\label{tau}
For each $i\in I$, fix a $u_i\in {\bf U}_{\alpha_i}\backslash\{1\}$. Let $w\in Z_J$, then we have
$$
\tau_i\dot{w}D(\theta)_J=\left\{
\begin{array}{ll}
 a_i\dot{w}D(\theta)_J-\dot{s_i}^{-1}\dot{w}D(\theta)_J &\ \mbox{if}~s_iw\leq w\\
b_i\dot{w}D(\theta)_J &\ \mbox{if}~s_iww_J<ww_J~\mbox{and}~s_iw\geq w\\
 0 &\ \mbox{if}~s_iww_J\geq ww_J
\end{array}\right.
$$
where $a_i=(\theta^w(\dot{s_i}h_i(u_i)\dot{s_i}))^{-1}$, and $b_i\in \Bbbk$ which depends on the choice of the representative of each element $w\in W$.
\end{lemma}

\begin{Cor}\label{tau2}
Let $j_1,\cdots,j_k\in I$. If the coefficient $($in terms of the basis given in Proposition \ref{basis}$)$ of $\dot{w_1}D(\theta)_J$ in $\tau_{j_k}\cdots\tau_{j_1}\dot{w_2}D(\theta)_J$ is nonzero, then $w_1=w_2$, or there exists a $1\leq t\leq k$ and a subset $\{i(1),i(2),\cdots,i(t)\}$ of $\{1,2,\cdots,k\}$ such that $(\op{i})$ $i(1)<i(2)<\cdots<i(t)$, $(\op{ii})$ $\ell(w_1)=\ell(w_2)-t$, and $(\op{iii})$ $w_1=s_{j_{i(t)}}\cdots s_{j_{i(1)}}w_2$.
\end{Cor}

\noindent As an easy consequence of Corollary \ref{tau2}, we have
\begin{Cor}\label{tau3}
Let $j_1,\cdots,j_k\in I$. Then

\noindent $(1)$ The coefficient of $\dot{w_1}D(\theta)_J$ in $\tau_{j_k}\cdots\tau_{j_1}\dot{w_2}D(\theta)_J$ is zero if $\ell(w_2)-\ell(w_1)>k$.

\noindent $(2)$ If $\ell(w_2)-\ell(w_1)=k$, then the coefficient of $\dot{w_1}D(\theta)_J$ in $\tau_{j_k}\cdots\tau_{j_1}\dot{w_2}D(\theta)_J$ is nonzero if and only if $w_1=s_{j_k}\cdots s_{j_1}w_2$.
\end{Cor}

Let $\theta$ be a one-dimensional character of $\bf T$, and $V$ be a $\Bbbk{\bf G}$-module which can also be regarded as a $\Bbbk{\bf T}$-module. Thus we denote by $$V_\theta=\{v\in V \mid tv=\theta(t)v, ~\forall t \in{\bf T}\}.$$
By restriction, one regards $\theta$ a character of $T_{q^a}$ and $V$ a $\Bbbk G_{q^a}$-module for any positive integer $a$. We denote by
$$V_{\theta,q^a}=\{v\in V \mid tv=\theta(t)v,\forall t\in T_{q^a}\}.$$ The following lemma is easy but useful in the main proof of this section.

\begin{lemma}\label{eigenvector}
Let $M$ be a $\Bbbk {\bf G}$-module and $N$ be a submodule of $M$. Assume that $\chi_1,\chi_2,\dots,\chi_m$ are different characters of ${\bf T}$ $($resp.~$T_{q^b}$ for some positive integer $b$$)$. If $\displaystyle \sum_{i=1}^m a_i\xi_i\in N$ with each $a_i\ne 0$ and $\xi_i\in M_{\chi_i}$ $($resp.~$\xi_i\in M_{\chi_i,q^b}$$)$, then $\xi_i\in N$ for $1\leq i\leq m$.
\end{lemma}
\begin{proof}The proof is obvious by induction on $m$.\end{proof}

Now we return to the main step of the proof. By Lemma \ref{basis}, it is sufficient to prove that $E(\theta)_J'$ is irreducible for any $J\subseteq I(\theta)$. This follows from the following four technical results.

\begin{lemma}\label{CDY-lemma}
Let $\theta\in\widehat{\bf T}$, and let $M$ be a $\Bbbk{\bf G}$-module and $\eta\in M_\theta$. If $ M'$ is a submodule of $  M$ containing $\underline{U_{q^a}}\eta$ for some positive integer $a$, then $\eta \in   M'$.
\end{lemma}

\begin{lemma}\label{claim1}
For $J\subseteq I(\theta)$, if $M$ is a nonzero submodule of $E(\theta)_J'$, then $$M\cap \sum_{w\in Z_J\cap W_\theta}\Bbbk \dot{w}D(\theta)_J\neq0.$$
\end{lemma}

\begin{lemma}\label{claim2}
For $J\subseteq I(\theta)$, if $M$ is a nonzero submodule of $E(\theta)_J'$, then $$\displaystyle M\cap \sum_{w\in Z_J\cap W_{I(\theta)}}\Bbbk \dot{w}D(\theta)_J\neq 0.$$
\end{lemma}

\begin{lemma}\label{claim3}
For $J\subseteq I(\theta)$, if $M$ is a nonzero submodule of $E(\theta)_J'$ such that $$\displaystyle M\cap \sum_{w\in Z_J\cap W_{I(\theta)}}\Bbbk \dot{w}D(\theta)_J\neq0,$$ then $D(\theta)_J\in M$ and hence $M=E(\theta)_J'$.
\end{lemma}
\noindent Once these lemmas are proved, we can prove Theorem \ref{main}.
\begin{proof}[{\it Proof of Theorem \ref{main}.}]
Combining Lemma \ref{claim2} and \ref{claim3}, we see that any nonzero submodule $M$ of $E(\theta)_J'$ contains $D(\theta)_J$, and hence $M=E(\theta)_J'$. In particular, all $E(\theta)_J'$ are irreducible for any $J\subset I(\theta)$. Therefore, all $E(\theta)_J$ are irreducible and pairwise non-isomorphic by Proposition \ref{EJ} and  \ref{basis}. This completes the proof.
\end{proof}

\noindent It remains to prove the above four lemmas.
\begin{proof}[{Proof of Lemma \ref{CDY-lemma}}] The proof is motivated by \cite[2.7]{Yang}. Let $$w_0=s_{\alpha_r}s_{\alpha_{r-1}}\dots s_{\alpha_1}$$ be a reduced
expression of the longest element $w_0$ of $W$, and set
$$\beta_i=s_{\alpha_1}s_{\alpha_{2}}\dots s_{\alpha_{i-1}
}{(\alpha_i)}.$$ Then for each $i=1,2, \dots, r$ and any positive integer $b$,
$$X_{i,q^b}=U_{\beta_r,q^b}U_{\beta_{r-1},q^b}\dots U_{\beta_i,q^b}$$
is a subgroup of ${U}_{q^b}= X_{1, q^b}$. Clearly, $X_{i,q^b}$ is a subgroup of
$X_{i,q^{b'}}$ if $\mathbb{F}_{q^b}$ is a subfield of
$\mathbb{F}_{q^{b'}}$. Here we set $X_{r+1,q^b}=\{1\}$.

First, we use induction on $i$ to show that there exists positive
integer $b_i$ such that the element $\underline {X_{i,q^{b_i}}}\eta$ is in $M'$. When $i=1$, this is true
for $b_1=a$ by assumption. Now we assume that $\underline {X_{i,q^{b_i}}}\eta$ is in $M'$. We show that
$\underline{X_{{i+1},q^{b_{i+1}}}}\eta$ is in $M'$ for
some positive integer  $b_{i+1}$.

Let $c_1,c_2,\dots,c_{q^{b_i}+1}$ be a complete set of
representatives of all cosets of $\mathbb{F}_{q^{b_i}}^*$ in
$\mathbb{F}_{q^{2b_i}}^*$. Choose $t_1,t_2,\dots,t_{q^{b_i}+1}\in {\bf T}$
such that $\beta_i(t_j)=c_j$ for $j=1,2,\dots,q^{b_i}+1$. Note that
$t\eta=\theta(t)\eta$ for any $t\in {\bf T}$.
Thus\begin{equation}\label{eq3}\sum_{j=1}^{q^{b_i}+1}\theta(t_j)^{-1} t_j
\underset{x\in U_{\beta_i,q^{b_i}}}{\sum} x\eta =
q^{b_i}\eta+ \underline{U_{\beta_i,q^{2b_i}}}\eta.\end{equation}
Since $X_{i,q^{b_i}}=X_{i+1,q^{b_i}}U_{\beta_i,q^{b_i}}$ and
$\underline{X_{i,q^{b_i}}} \eta$ is in $M'$, we see
that
$$
\aligned \zeta &\ := \sum_{j=1}^{q^{b_i}+1}\theta(t_j)^{-1} t_j\underset{x\in X_{i,q^{b_i}}}{\sum}x\eta
= \sum_{j=1}^{q^{b_i}+1}\theta(t_j)^{-1} t_j \underset{y\in X_{i+1,q^{b_i}}}{\sum}y
\underset{x\in U_{\beta_i,q^{b_i}}}{\sum}x\eta \\
&\ = \underset {y\in
 X_{i+1,q^{b_i}}}{\sum}\sum_{j=1}^{q^{b_i}+1}t_jyt_j^{-1}(\theta(t_j)^{-1}t_j\underset{x\in U_{\beta_i,q^{b_i}}}{\sum}x\eta) \in   M'.
\endaligned
$$
Choose $b_{i+1}$ such that all $\beta_m(t_j)(r\geq m \geq i)$ are
contained in $\mathbb{F}_{q^{b_{i+1}}}$ and
$\mathbb{F}_{q^{b_{i+1}}}$ contains  $\mathbb{F}_{q^{2b_i}}$. Then
$t_j y t_j^{-1}$ is in $X_{i+1,q^{b_{i+1}}}$ for any $y\in
X_{i+1,q^{b_i}}$. Then we have
\begin{equation} \label{eq4}
\begin{split}
\underline{X_{i+1,q^{b_{i+1}}}}  \zeta &\ = q^{(r-i)b_i}\underline{X_{i+1,q^{b_{i+1}}}} \sum_{j=1}^{q^{b_i}+1} \theta(t_j)^{-1} t_j \underset {x\in U_{\beta_i,q^{b_i}}}{\sum}x\eta\\
&\ = q^{(r-i)b_i}\underline{X_{i+1,q^{b_{i+1}}}}(q^{b_i}\eta + \underline{U_{\beta_i,q^{2b_i}}}\eta)\ \in
M'.
\end{split}
\end{equation}
Because $\underline{X_{i,q^{b_i}}}\eta$ is in $M'$,
we have $\underline{X_{i,q^{2b_i}}}\eta \in M'$. Thus
\begin{equation} \label{eq5}
\underline{X_{i+1,q^{b_{i+1}}}}\  \underline{X_{i,q^{2b_i}}} \eta=q^{2(r-i)b_i}\underline{X_{i+1,q^{b_{i+1}}}} \  \underline{U_{\beta_i,q^{2b_i}}}\eta \ \in   M'.\end{equation}

Since $q\neq 0$ in $\Bbbk$, combining formula (\ref{eq4}) and (\ref{eq5}) we
see that $\underline{X_{i+1,q^{b_{i+1}}}} \eta \in   M'$. Noting that
$X_{r,q^{b_r}}=U_{\beta_r,q^{b_r}}$, now we have $\underline{U_{\beta_r,q^{b_r}}}\eta \in   M'$ and $\underline{U_{\beta_r,q^{2b_r}}} \eta \in   M'$. Applying formula (\ref{eq3}) to
the case $i=r$ we get that $q^{b_r}\eta + \underline{U_{\beta_r,q^{2b_r}}}\eta \in M'$. Therefore $\eta$ is in $M'$. The lemma is proved.
\end{proof}

\begin{proof}[{\it Proof of Lemma \ref{claim1}}]
The proof is analogous to that of \cite[Claim 1]{CD1}. Assume that $M$ is a nonzero submodule of $E(\theta)_J'$ and $0\neq x\in M$. Then $x\in E(\theta)_{J,q^a}'=\Bbbk G_{q^a}D(\theta)_J$ for some $a>0$.

For any $K\subset I(\theta)$, set $\mathbb{M}(\theta,K)_{q^a}=\op{Ind}_{P_{K,q^a}}^{G_{q^a}}\Bbbk_{\theta}$. For a finite-dimensional $\Bbbk G_{q^a} $-module $N$, we denote by $N^*$ the dual space of $N$ which is also a $\Bbbk G_{q^a} $-module. Now assume that $L\in\op{Irr}_\Bbbk(G_{q^a})$ is a simple submodule of $\mathbb{M}(\theta,K)_{q^a}$. Then
$$\op{Hom}_{G_{q^a}}(\mathbb{M}(\theta^{-1},K)_{q^a}, L^*)=\op{Hom}_{G_{q^a}}((\mathbb{M}(\theta,K)_{q^a})^*, L^*)\ne 0.$$
Since for each character $\theta$, $\op{Ind}_{P_{K,q^a}}^{G_{q^a}}\Bbbk_{\theta}$ is a quotient module of $\op{Ind}_{B_{q^a}}^{G_{q^a}}\Bbbk_{\theta}$, which is a quotient module of $\op{Ind}_{U_{q^a}}^{G_{q^a}}\op{tr}$, we have $\op{Hom}_{G_{q^a}}(\op{Ind}_{U_{q^a}}^{G_{q^a}}\op{tr}, L^*)\ne 0$. Therefore
$$(L^*)^{U_{q^a}}\cong \op{Hom}_{U_{q^a}}(\op{tr}, L^*)\cong   \op{Hom}_{G_{q^a}}(\op{Ind}_{U_{q^a}}^{G_{q^a}}\op{tr}, L^*)\ne 0$$ by Frobenius reciprocity. When $\op{char}\Bbbk\neq\op{char}\bar{\mathbb{F}}_q$,
$(L^*)^{U_{q^a}}\ne 0$ is  equivalent to $L^{U_{q^a}}\ne 0$ since $U_{q^a}$ acts semisimply on $L$.

It is clear that $(E(\theta)_{J,q^a}')^{U_{q^a}}\subset\bigoplus_{w\in Z_J}\Bbbk\underline{U_{w_Jw^{-1},q^a}}\dot{w}D(\theta)_J$ by Proposition \ref{basis}, and there is a simple module  $L\in\op{Irr}_\Bbbk(G_{q^a})$ such that $L\subset\Bbbk G_{q^a}x\subset E(\theta)_{J,q^a}'\cap M$. By the previous paragraph, $L^{U_{q^a}}\neq0$, which implies that $(E(\theta)_{J,q^a}')^{U_{q^a}}\cap M\neq0$. Assume that
$$0\neq\xi=\sum_{w\in Z_J}c_w\underline{U_{w_Jw^{-1},q^a}}\dot{w}D(\theta)_J\in(E(\theta)_{J,q^a}')^{U_{q^a}}\cap M,\quad c_w\in\Bbbk.$$
Notice that
$$
\aligned
\underline{U_{q^a}}\xi &\ =\underline{U_{q^a}}\sum_{w\in Z_J}c_w\underline{U_{w_Jw^{-1},q^a}}\dot{w}D(\theta)_J\\
&\ =\sum_{w\in Z_J}\underline{U_{w_Jw^{-1},q^a}'}\cdot\underline{U_{w_Jw^{-1},q^a}}c_w\underline{U_{w_Jw^{-1},q^a}}\dot{w}D(\theta)_J\\
&\ =\sum_{w\in Z_J}\underline{U_{w_Jw^{-1},q^a}'}\cdot\underline{U_{w_Jw^{-1},q^a}}c_wq^{a\ell(w_Jw^{-1})}\dot{w}D(\theta)_J\\
&\ =\underline{U_{q^a}}\sum_{w\in Z_J}c_wq^{a\ell(w_Jw^{-1})}\dot{w}D(\theta)_J \in M.
\endaligned
$$
It follows that $\displaystyle 0\ne\underline{U_{q^a}}\sum_{w\in W_\theta}c'_{w}\dot{v}\dot{w}D(\theta)_J\in M_{\theta^v,q^a}$ for some $v\in W$ by Lemma \ref{eigenvector}.
Since $\displaystyle \sum_{w\in W_\theta}c'_{w}\dot{v}\dot{w}D(\theta)_J\in(E(\theta)_J')_{\theta^v}$, we see that $\displaystyle \sum_{w\in W_\theta}c'_{w}\dot{v}\dot{w}D(\theta)_J \in M$ by Lemma \ref{CDY-lemma}. Thus we have
$$0\ne \sum_{w\in W_\theta}c'_{w}\dot{w}D(\theta)_J \in (E(\theta)_J')_{\theta}\cap M.$$
This completes the proof.
\end{proof}

\begin{proof}[{\it Proof of Lemma \ref{claim2}}]
We first make some preliminaries. Let $W_{J(\theta)}$ be the minimal parabolic subgroup containing $W_\theta$. Then $J(\theta)\supset I(\theta)$ by Lemma \ref{Wtheta}. Let $\{v_1=e,v_2,\cdots,v_n\}$ be a complete set of the representatives (with minimal length) of the left cosets of $W_{I(\theta)}$ in $W_{J(\theta)}$. We set
$$V_i=\displaystyle\sum\limits_{{w\in W_{I(\theta)}}\atop{v_iw\in Z_J}}\Bbbk\dot{v_i}\dot{w}D(\theta)_J  ~~~\text{and}~~~ V=\displaystyle\sum\limits_{1\leq i\leq n}V_i.$$
As a ${\bf T}$-module, each $V_i$ is a ${\bf T}$-weight  space of weight $\theta^{v_i}$ .
One should be careful that $\theta^{v_i}$ may equals to $\theta^{v_j}$ even $v_i\ne v_j$.
Clearly, any $\varepsilon\in V$ can be written uniquely as $\varepsilon=\displaystyle\sum_i\varepsilon_i$ with $\varepsilon_i\in V_i$. Thus for each element $\varepsilon\in V$, define $N(\varepsilon)=\{v_i\mid\varepsilon_i\neq0\}$. On the other hand, each element $\xi \in V$ has a unique decomposition
$$\xi= \xi_{\chi_1}+ \xi_{\chi_2}+ \dots + \xi_{\chi_s},$$
where $\chi_1, \chi_2, \dots , \chi_s$ are different characters. In this case we call $\xi_{\chi_k}$ the $\chi_k$-weight factor of $\xi$ for $k=1,2,\dots, s$.

We fix $u_i\in {\bf U}_{\alpha_i}\backslash\{1\}$ for each $i\in I$  and then consider the functor $\tau_i$ as in Lemma \ref{tau}. Given $\chi\in\widehat{\bf T}$ (actually, $\chi= \theta^{v_k}$ for some integer $k$ in our case) and an element $\varepsilon \in V_\chi$, motivated by the formula in Lemma \ref{tau}, for each $i\in I$, we denote by $\Omega^{\chi}_i(\varepsilon)$ the $\chi^{s_i}$-weight factor of the element $\tau_i\varepsilon-(\chi(\dot{s_i}h_i(u_i)\dot{s_i}))^{-1}\varepsilon$. Noting that when the weight of $\varepsilon$ is known, we can simply denote by $\Omega_i(\varepsilon)= \Omega^{\chi}_i(\varepsilon)$.
Clearly, $\Omega_i$ induces a linear operator on $M\cap V$ by Lemma \ref{eigenvector}. Then it is not difficult to see that
$|N(\Omega_i(\varepsilon))| \leq |N(\varepsilon)|$ using Lemma \ref{tau}.

Now we return to the main proof. By Lemma \ref{claim1}, it is enough to show that if $0\ne\varepsilon\in M\cap V$, then $$\displaystyle M\cap \sum_{w\in Z_J\cap W_{I(\theta)}}\Bbbk \dot{w}D(\theta)_J\neq 0.$$ We will show this by induction on $|N(\varepsilon)|$.

Using Lemma \ref{eigenvector}, we can assume that $\varepsilon$ is a ${\bf T}$-weight vector without lost of generality.  If $|N(\varepsilon)|=1$, then
$$0\ne \dot{v_i}^{-1}\varepsilon\in M\cap \sum_{w\in Z_J\cap W_{I(\theta)}}\Bbbk \dot{w}D(\theta)_J$$
for some $v_i$ and we are done. Now we assume that $|N(\varepsilon)|>1$. Choose $v_j\in N(\varepsilon)$ with $\ell(v_j)=\min\{\ell(v_i)\mid v_i\in N(\varepsilon)\}$ and let $v_j=s_{j_1}\cdots s_{j_t}$ be its one reduced expression. Now we let $\varepsilon' = \Omega_{j_t}\cdots\Omega_{j_1}(\varepsilon)$ which is nonzero and it is not difficult to see that $\varepsilon' \in M\cap V_\theta$. Moreover, we have $|N(\varepsilon')|\leq |N(\varepsilon)|$ by Lemma \ref{tau} and previous discussion.  Now we set $\varepsilon'=\xi+\zeta$, where $0\ne \xi\in V_1$ and $\zeta\in\sum_{i\geq2}V_i$.
If $\zeta =0$, we are done. Otherwise we choose $v_k\in N(\zeta)$ with $\ell(v_k)=\min\{\ell(v_i)\mid v_i\in N(\zeta)\}$ and let $v_k=s_{k_1}\cdots s_{k_r}$ be its one reduced expression. By the property of $W_{I(\theta)}$ and Lemma \ref{tau}, we have (1) $r\geq 2$; (2) $0\ne \varepsilon'':=\Omega_{k_{r-1}}\cdots\Omega_{k_1}(\varepsilon')\in M\cap V_{\theta}$; (3) $s_{k_l}\not\in W_{I(\theta)}$ for some $1\leq l\leq r-1$. It is clear that $$|N(\Omega_{k_i}\cdots\Omega_{k_1}(\varepsilon'))|\leq|N(\Omega_{k_{i-1}}\cdots\Omega_{k_1}(\varepsilon'))|$$ for any $1\leq i\leq r-1$ and
$$|N(\Omega_{k_l}\cdots\Omega_{k_1}(\varepsilon'))|<|N(\varepsilon')|$$
by the choice of $l$ and Lemma \ref{tau} (since the term in $V_1$ is killed by $\Omega_{k_l}\cdots\Omega_{k_1}$). It follows that $|N(\varepsilon'')|<|N(\varepsilon')|\leq|N(\varepsilon)|$ and the result follows from applying the induction hypothesis to $\varepsilon''$. This completes the proof.
\end{proof}

\begin{proof}[{\it Proof of Lemma \ref{claim3}}]
The proof is identical to that of \cite[Claim 2]{CD1} as long as we replace $Y_J$ there with $Z_J\cap W_{I(\theta)}$, and $D_J$ there with $D_J(\theta)$, and using Lemma \ref{tau}, Corollary \ref{tau2}, Corollary \ref{tau3}.
\end{proof}

\section{The natural characteristic (antidominant case) }

 We consider the natural characteristic case in this and next sections. In this section, we always assume that $\Bbbk=\bar{\mathbb{F}}_q$ and $\theta\in X({\bf T})$, the group of rational characters of ${\bf T}$.

For each $\alpha\in\Phi$, we fix an isomorphism $\varepsilon_\alpha: \bar{\mathbb{F}}_q\rightarrow{\bf U}_\alpha$ such that $t\varepsilon_\alpha(c)t^{-1}=\varepsilon_\alpha(\alpha(t)c)$ for any $t\in{\bf T}$ and $c\in\bar{\mathbb{F}}_q$. Set $U_{\alpha,q^a}=\varepsilon_\alpha(\mathbb{F}_{q^a})$. For each $i\in I$, we fix a homomorphism  $\varphi_i: SL_2(\bar{\mathbb{F}}_q)\rightarrow{\bf G}_i$ such that
$$\varphi_i\left(\begin{array}{cc}1 &\ a\\0 &\ 1\end{array}\right)=\varepsilon_{\alpha_i}(a),~~\varphi_i\left(\begin{array}{cc}1 &\ 0\\a &\ 1\end{array}\right)=\varepsilon_{-\alpha_i}(a)$$
for any $a\in\bar{\mathbb{F}}_q$.
For each $c\in\Bbbk^*$ and $i\in I$, one denote
$$h_i(c):=\varphi_i\left(\begin{array}{cc}c &\ 0\\0 &\ c^{-1}\end{array}\right),~~\dot{s_i}=\varphi_i\left(\begin{array}{cc}0 &\ 1\\-1 &\ 0\end{array}\right).$$
It is easy to check that
\begin{equation}\label{s-1us}
\dot{s_i}^{-1}\varepsilon_{\alpha_i}(c)\dot{s_i}=\varepsilon_{\alpha_i}(-c^{-1})\dot{s_i}h_i(c)\varepsilon_{\alpha_i}(-c^{-1})
\end{equation}
for any $c\in\Bbbk^*$.

Let $Y({\bf T})$ be the set of algebraic group homomorphisms $\Bbbk^*\rightarrow{\bf T}$. There is a pair $\langle-,-\rangle: X({\bf T})\times Y({\bf T})\rightarrow\mathbb{Z}$ such that for any $\lambda\in X$ and $\mu\in Y$, $\lambda(\mu(t))=t^{\langle\lambda,\mu\rangle}$ $(t\in\Bbbk^*)$ (see. \cite[Part II 1.3]{Jan1}). For any $\alpha\in\Phi$, denote $\alpha^\vee\in Y({\bf T})$ its dual root such that $\langle\alpha,\alpha^\vee\rangle=2$.

For $\theta\in X({\bf T})$, we call $\theta$ dominant (resp. antidominant) if $\langle\theta,\alpha^\vee\rangle\ge 0$ (resp. $\langle\theta,\alpha^\vee\rangle\le 0$) for all $\alpha\in\Delta$. In other words, $\theta$ is antidominant if and only if $-\theta$ is dominant. Throughout this section, we assume that $\theta$ is antidominant. The main theorem of this section is
\begin{theorem}\label{antidominant}
If $\theta$ is antidominant, then  all $\Bbbk {\bf G}$-modules $E(\theta)_J$ are irreducible for  $J\subset I(\theta)$. Consequently, the $\Bbbk {\bf G}$-module $\mathbb{M}(\theta)$ has exactly $2^{|I(\theta)|}$ pairwise nonisomorphic composition factors, each occurring with multiplicity $1$.
\end{theorem}

Before proving Theorem \ref{antidominant}, we make some preliminaries (from Lemma \ref{powsum} to Lemma \ref{key}). The following two elementary lemmas are well known, and will be frequently used in this and next section.
\begin{lemma}[{\cite[Lemma 2.1]{Se2}}]\label{powsum}
Let ${\bf p}_k(\mathbb{F}_{q^a})=\sum\limits_{t\in\mathbb{F}_{q^a}}t^k$ and ${\bf p}_k(\mathbb{F}_{q^a}^*)=\sum\limits_{t\in\mathbb{F}_{q^a}^*}t^k$. Then
$${\bf p}_k(\mathbb{F}_{q^a})=\left\{
\begin{array}{ll}
-1 &\ \mbox{if}~(q^a-1)|k~\mbox{and}~k\neq0\\
0 &\ \mbox{otherwise}
\end{array}
\right.,~~~{\bf p}_k(\mathbb{F}_{q^a}^*)=\left\{
\begin{array}{ll}
-1 &\ \mbox{if}~(q^a-1)|k\\
0 &\ \mbox{otherwise}
\end{array}
\right..$$
\end{lemma}

\begin{lemma}[{\cite[5.1]{Haboush}}]\label{Binomial}
Assume that $p$ is a prime. Let $m, n$ be two positive integers with $p$-adic expansion
$$m=\sum_{i}a_i p^i~~~~~ \text{and} ~~~~~n=\sum_{i}b_i p^i.$$
Then we have
$$\left(m\atop n\right)\equiv \prod_{i}\left(a_i\atop b_i\right) \quad(\op{mod}p) .$$
In particular, $p|\left(m\atop n\right)$ if and only if there exists $i$ such that
$a_i <b_i$.
\end{lemma}

\begin{lemma}\label{A1}
Let $i\in I$, and let $M$ be a $\Bbbk{\bf G}$-module and $0\neq\xi\in M^{{\bf U}_{\alpha_i}}$. Assume that there is an integer $0<m<q^a$ such that $h_i(t)\xi=t^{-m}\xi$ for any $t\in\Bbbk^*$. Then $\Bbbk{\bf G}\xi=\Bbbk{\bf G}\underline{U_{\alpha_i,q^a}}\dot{s_i}\xi$.
\end{lemma}
\begin{proof}
Let $V=\Bbbk{\bf G}\underline{U_{\alpha_i,q^a}}\dot{s_i}\xi$ and $v_1:=\underline{U_{\alpha_i,q^a}}\dot{s_i}\xi$. We have to show that $\xi\in V$. Multiplying $v_1$ by the sum of the representatives of the cosets of $U_{\alpha_i,q^a}$ in $U_{\alpha_i,q^{2a}}$, we obtain $v_2:=\underline{U_{\alpha_i,q^{2a}}}\dot{s_i}\xi\in V$, and hence
\begin{equation}\label{siv}
\dot{s_i}^{-1}v_2=\xi+\sum_{t\in\mathbb{F}_{q^{2a}}^*}t^{-m}\varepsilon_{\alpha_i}(-t^{-1})\dot{s_i}\xi=
\xi+(-1)^m\sum_{t\in\mathbb{F}_{q^{2a}}^*}t^m\varepsilon_{\alpha_i}(t)\dot{s_i}\xi\in V
\end{equation}
by (\ref{s-1us}). Choose a square root $c_t\in\mathbb{F}_{q^{4a}}^*$ for each $t\in\mathbb{F}_{q^{2a}}^*$. Observe that
$$
\aligned
\sum_{t\in\mathbb{F}_{q^{2a}}^*}t^{q^a-1}c_t^{-m}h_i(c_t)v_1
&\ =\sum_{{b\in\mathbb{F}_{q^a}}\atop{t\in\mathbb{F}_{q^{2a}}^*}}t^{q^a-1}c_t^{-m}h_i(c_t)\varepsilon_{\alpha_i}(b)\dot{s_i}\xi\\
&\ =\sum_{{b\in\mathbb{F}_{q^a}}\atop{t\in\mathbb{F}_{q^{2a}}^*}}t^{q^a-1}\varepsilon_{\alpha_i}(bt)\dot{s_i}\xi\\
&\ ={\bf p}_{q^a-1}(\mathbb{F}_{q^{2a}}^*)\dot{s_i}\xi+\sum_{{b\in\mathbb{F}_{q}^*}\atop{t\in\mathbb{F}_{q^{2a}}^*}}t^{q^a-1}\varepsilon_{\alpha_i}(bt)\dot{s_i}\xi\\
&\ =-\dot{s_i}\xi+ \sum_{c\in\mathbb{F}_{q^{2a}}^*}{\bf p}_{q^a-1}(\mathbb{F}_{q^a}^*)c^{q^a-1}\varepsilon_{\alpha_i}(c)\dot{s_i}\xi\\
&\ =-\sum_{c\in\mathbb{F}_{q^{2a}}}c^{q^a-1}\varepsilon_{\alpha_i}(c)\dot{s_i}\xi.
\endaligned
$$
It follows that $v_3:=\displaystyle\sum_{c\in\mathbb{F}_{q^{2a}}}c^{q^a-1}\varepsilon_{\alpha_i}(c)\dot{s_i}\xi\in V$. Let $\mu_k=\displaystyle\sum_{b\in\mathbb{F}_{q^{2a}}}b^k\varepsilon_{\alpha_i}(b)\dot{s_i}\xi$. Similarly, we have
\begin{equation}\label{tctau}
\sum_{t\in\mathbb{F}_{q^{2a}}^*}t^mc_t^{-m}h_i(c_t)\mu_k
=\sum_{{b\in\mathbb{F}_{q^{2a}}}\atop{t\in\mathbb{F}_{q^{2a}}^*}}b^kt^m\varepsilon_{\alpha_i}(bt)\dot{s_i}\xi
=\sum_{c\in\mathbb{F}_{q^{2a}}}{\bf p}_{k-m}(\mathbb{F}_{q^{2a}}^*)c^m\varepsilon_{\alpha_i}(c)\dot{s_i}\xi,
\end{equation}
which is nonzero ($=-\mu_m$) if and only if $k=m$ by Lemma \ref{powsum}.
It follows that
$$\sum_{t\in\mathbb{F}_{q^{2a}}^*}t^mc_t^{-m}h_i(c_t)\varepsilon_{\alpha_i}(-1)v_3
=-\left({q^a-1}\atop m\right)\mu_m$$
by (\ref{tctau}) and Lemma \ref{powsum}, and hence $\mu_m\in V$ since $\displaystyle\left({q^a-1}\atop m\right)\neq0~(\op{mod}p)$ by Lemma \ref{Binomial} and our assumption. Combining this and (\ref{siv}) yields $\xi\in V$ which completes the proof.
\end{proof}

\begin{remark}\label{remA1}
{\rm More precisely, the proof of Lemma \ref{A1} tells us that $$\xi\in\Bbbk G_{i,q^{4a}}\underline{U_{\alpha_i,q^a}}\dot{s_i}\xi$$ $($keeping the notation in Lemma \ref{A1}$)$. On the contrary,  $\xi\not\in\Bbbk G_{i,q^a}\underline{U_{\alpha_i,q^a}}\dot{s_i}\xi$ in general. For example, let ${\bf G}=SL_2(\bar{\mathbb{F}}_q)$ and $\theta$  be a nontrivial character of ${\bf T}$. Let $\xi={\bf 1}_\theta$. Then $\underline{U_{q^a}}\dot{s}\xi$ only generates the (unique) simple socle of $\Bbbk G_{q^a}\xi$ as $\Bbbk G_{q^a}$-modules by \cite[Theorem 4.6]{YY}.}
\end{remark}

Similar to \cite[Lemma 4.5]{CD2}, we have the following key lemma whose proof is identical to that of \cite[Lemma 4.5]{CD2} as long as one replace $C_J$ there with $D(\theta)_J$.
\begin{lemma}\label{key}
Let $w\in Z_J$ and $A=\{\alpha_1,\alpha_2,\dots, \alpha_m\}$ and $B=\{\beta_1,\beta_2,\dots, \beta_n\}$ be two disjoint subsets of $\Phi_{w_Jw^{-1}}^-$, and assume that $\sum_il_i\alpha_i\in A$ whenever $\sum_il_i\alpha_i\in\Phi$ for some $l_i\in\mathbb{Z}_{\geq0}$.
Let $a<b$ be integers with $a|b$, and denote
$$\delta:=\underline{U_{\alpha_1,q^b}}\cdots\underline{U_{\alpha_m,q^b}}\cdot
\underline{U_{\beta_1,q^a}}\cdots\underline{U_{\beta_n,q^a}}\dot{w}D(\theta)_J.$$

\noindent We have

\noindent$\op{(i)}$ Assume that $k\beta_1+\sum_il_i\alpha_i\in A$ whenever $k\beta_1+\sum_il_i\alpha_i\in\Phi$ for some $k\in\mathbb{Z}_{>0}$ and $l_i\in\mathbb{Z}_{\geq0}$. Then
$$x\delta=\underline{U_{\alpha_1,q^b}}\cdots\underline{U_{\alpha_m,q^b}}\cdot
x \underline{U_{\beta_1,q^a}}\cdots\underline{U_{\beta_n,q^a}}\dot{w}D(\theta)_J$$
for any $x\in U_{\beta_1,q^b}$.

\noindent$\op{(ii)}$ Let $\gamma\in\Phi_{w_Jw^{-1}}^+$. Assume that $k\gamma+\sum_il_i\alpha_i+\sum_im_i\beta_i\in A$ whenever $k\gamma+\sum_il_i\alpha_i+\sum_im_i\beta_i\in\Phi_{w_Jw^{-1}}^-$ for some $k\in\mathbb{Z}_{>0}$ and $l_i,m_i\in\mathbb{Z}_{\geq0}$. Then $y\delta=\delta$ for any $y\in U_{\gamma,q^b}$.
\end{lemma}

Now we return to the main step of the proof. The Theorem \ref{antidominant} follows from the following two technical results.
\begin{lemma}\label{c1}
Let $M$ be a nonzero submodule  of $E(\theta)_J'$. Then $$\underline{U_{w_Jw^{-1},q^c}}\dot{w}D(\theta)_J\in M$$ for some $w\in Z_J$ and $c>0$.
\end{lemma}

\begin{lemma}\label{c2}
Let $k\in I$ and $w\in Z_J$, and assume that $s_kw\in Z_J$ and $s_kw>w$. Then
$\underline{U_{w_Jw^{-1},q^{ah}}}\dot{w}D(\theta)_J\in\Bbbk{\bf G}\underline{U_{w_Jw^{-1}s_k,q^a}}\dot{s_k}\dot{w}D(\theta)_J$ for some $h>0$.
\end{lemma}

\noindent Before proving these lemmas we show how to deduce Theorem \ref{antidominant} from them.
\begin{proof}[Proof of Theorem \ref{antidominant}.]
Let $J\subset I(\theta)$, and $N$ be a nonzero submodule of $E(\theta)_J'$. By Lemma \ref{c1} we have $\underline{U_{w_Jw^{-1},q^a}}\dot{w}D(\theta)_J\in N$ for some $w\in Z_J$ and $a>0$. Applying Lemma \ref{c2} repeatedly yields $\underline{U_{w_J,q^b}}D(\theta)_J\in N$ for some $b>0$.
By \cite[Lemma 2]{St},  since $\op{char}\Bbbk=\op{char}\mathbb{F}_q$, we have
$$\sum_{w\in W_J}(-1)^{\ell(w)}\dot{w}\underline{U_{w_J,q^b}}D(\theta)_J=\sum_{w\in W_J}q^{b\ell(w)}D(\theta)_J= D(\theta)_J \in N.$$
It follows that $N=E(\theta)_J'$ which implies the irreducibility of $E(\theta)_J'$. Finally, Lemma \ref{basis} implies the irreducibility of $E(\theta)_J$.
\end{proof}

\noindent It remains to prove Lemma \ref{c1} and \ref{c2}.
\begin{proof}[{Proof of Lemma \ref{c1}}]
Assume that $M$ is a nonzero submodule of $E(\theta)_J'$ and $0\neq x\in M$. Then $x\in\Bbbk G_{q^a}D(\theta)_J$ for some $a>0$. It is clear that
$$(\Bbbk G_{q^a}x)^{U_{q^a}}\subset(\Bbbk G_{q^a}D(\theta)_J)^{U_{q^a}}\subset\bigoplus_{w\in Z_J}\Bbbk\underline{U_{w_Jw^{-1},q^a}}\dot{w}D(\theta)_J$$
by Proposition \ref{basis}. Moreover, $(\Bbbk G_{q^a}x)^{U_{q^a}}\neq0$ by \cite[Proposition 26]{Se}. That is, some nonzero element
$$\xi=\sum_{w\in Z_J}c_w\underline{U_{w_Jw^{-1},q^a}}\dot{w}D(\theta)_J\in(\Bbbk G_{q^a}x)^{U_{q^a}}\subset M,\quad c_w\in\Bbbk.$$
Let $Y_{\xi}=\{w\in Z_J\mid c_w\neq0\}$. Let $\Phi_{\xi}=\bigcup_{w\in Y_\xi}\Phi_{w_Jw^{-1}}^-$. We fix an order in $\Phi_\xi$ such that $\Phi_\xi=\{\beta_1,\cdots,\beta_m\}$ with $\op{ht}(\beta_1)\geq\cdots\geq\op{ht}(\beta_m)$.

Let $b > a$ be an integer such that $a|b $. For each $w\in Y_\xi$, write $\Phi_{w_Jw^{-1}}^-=\{\gamma_1,\cdots,\gamma_t\}$ with the order inherited from $\Phi_\xi$ (In particular, $\op{ht}(\gamma_1)\geq\cdots\geq\op{ht}(\gamma_t)$). For any $0\leq d\leq t$, set
$$\Theta(w,d,b,a):=\underline{U_{\gamma_1,q^b}}\cdots\underline{U_{\gamma_d,q^b}}\cdot
\underline{U_{\gamma_{d+1},q^a}}\cdots\underline{U_{\gamma_t,q^a}},$$
if $d>0$ and
$$\Theta(w,0,b,a):=\underline{U_{\gamma_1,q^a}}\cdots\underline{U_{\gamma_t,q^a}}.$$
We need the following claim whose proof is identical to that of \cite[Proposition 4.3]{CD2},as long as one replace $C_J$ there with $D(\theta)_J$.

\smallskip
\noindent{\it Let $Y$ be a nonempty subset of $Y_{\xi}$ and $\Phi_Y=\bigcup_{w\in Y}\Phi_{w_Jw^{-1}}^-=\{\alpha_1,\cdots,\alpha_n\}$ with the order inherited from $\Phi_\xi$ $($In particular $\op{ht}(\alpha_1)\geq\cdots\geq\op{ht}(\alpha_n)$$)$, and let $d\geq 0$ be an integer such that $\alpha_1,\dots,\alpha_d\in\bigcap_{w\in Y}\Phi_{w_Jw^{-1}}^-$. If $$\xi_d:=\sum_{w\in Y}c_w\Theta(w,d,b,a)\dot{w}D(\theta)_J\in M$$ for some $b\neq a$ and $a|b$, then $\underline{U_{w_Jw^{-1},q^b}}\dot{w}D(\theta)_J\in M$ for some $w\in Y$}.

\smallskip
The lemma follows immediately from applying the above claim to $Y=Y_\xi$, $d=0$ and $\xi=\xi_d$.
\end{proof}

\begin{proof}[{Proof of Lemma \ref{c2}}]
Let $V=\Bbbk{\bf G}\underline{U_{w_Jw^{-1}s_k,q^a}}\dot{s_k}\dot{w}D(\theta)_J$ and set
\begin{equation}\label{v}
v:=\underline{U_{w_Jw^{-1}s_k,q^a}}\dot{s_k}\dot{w}D(\theta)_J=
\underline{U_{\alpha_k,q^a}}\cdot\underline{(U_{w_Jw^{-1},q^a})^{s_k}}\dot{s_k}\dot{w}D(\theta)_J,
\end{equation}
where the superscript $s_k$ means the subgroup is conjugated by $s_k$.
Multiplying $v$ by the sum of representatives of left cosets of $(U_{w_Jw^{-1},q^a})^{s_k}$ in $(U_{w_Jw^{-1},q^{2a}})^{s_k}$, we have
\begin{equation}\label{any}
\underline{U_{\alpha_k,q^a}}\cdot\underline{(U_{w_Jw^{-1},q^{2a}})^{s_k}}\dot{s_k}\dot{w}D(\theta)_J\in V
\end{equation}
since $\underline{U_{\alpha_k,q^a}}\cdot\underline{(U_{w_Jw^{-1},q^{2a}})^{s_k}}=\underline{(U_{w_Jw^{-1},q^{2a}})^{s_k}}\cdot\underline{U_{\alpha_k,q^a}}$.
Consider the following two cases:

\noindent{\bf Case 1}: $({\bf U}_{w_Jw^{-1}})^{s_k}={\bf U}_{w_Jw^{-1}}$.

\noindent In this case we have $s_kww_J=ww_Js_l$ (equivalently, $w_Jw^{-1}(\alpha_k)=\alpha_l$) for some $l\in I\backslash I(\theta)$ (since $s_kw\in Z_J$). Hence, $\langle\theta,\alpha_l^\vee\rangle<0$. Since $\underline{U_{w_Jw^{-1}s_k,q^{a'}}}\dot{s_k}\dot{w}D(\theta)_J\in V$ if $a|a'$, we can assume that $-\langle\theta,\alpha_l^\vee\rangle<q^a$ without loss of generality. Moreover, we have
\begin{equation}\label{v2}
v=\underline{U_{w_Jw^{-1},q^a}}\cdot\underline{U_{\alpha_k,q^{a}}}\dot{s_k}\dot{w}D(\theta)_J.
\end{equation}
Multiplying $v$ by the sum of representatives of all the left cosets of $U_{w_Jw^{-1},q^a}$ in $\underline{U_{w_Jw^{-1},q^{4a}}}$ to (\ref{v2}), one obtain
\begin{equation}\label{v3}
v_1:=\underline{U_{w_Jw^{-1},q^{4a}}}\cdot\underline{U_{\alpha_k,q^a}}\dot{s_k}\dot{w}D(\theta)_J\in V.
\end{equation}
Noting that when $w'\in W_J$ and $J\subset J(\theta)$, $\theta^{w'}=\theta=\theta^{w_J}$ (Recall the notation in (\ref{theta^w})). So, $h_k(t)ww'{\bf 1}_\theta=t^{\langle \theta^{ww'},\alpha_k^\vee\rangle}ww'{\bf 1}_\theta=t^{\langle \theta^{ww_J},\alpha_k^\vee\rangle}ww'{\bf 1}_\theta$. Since
$$h_k(t)\dot{w}D(\theta)_J=t^{\langle \theta^{ww_J},\alpha_k^\vee\rangle}\dot{w}D(\theta)_J=t^{\langle \theta,w_Jw^{-1}(\alpha_k)^\vee\rangle}\dot{w}D(\theta)_J=t^{\langle \theta,\alpha_l^\vee\rangle}\dot{w}D(\theta)_J,$$ and $U_{w_Jw^{-1},q^{4a}}$ is invariant under $G_{k,q^{4a}}$-conjugation, it follows that
$$\underline{U_{w_Jw^{-1},q^{4a}}}\dot{w}D(\theta)_J\in V$$
combining Lemma \ref{A1}, Remark \ref{remA1}, and (\ref{v3}). This completes the proof.

\noindent{\bf Case 2}: $({\bf U}_{w_Jw^{-1}})^{s_k}\neq{\bf U}_{w_Jw^{-1}}$.

\noindent The idea is similar to the proof of \cite[Proposition 4.4]{CD2}. However the discussion is more complicated. By (\ref{s-1us}), we have
$$
\aligned
\dot{s_k}^{-1}v &\ = \underline{U_{-\alpha_k,q^a}}\cdot\underline{U_{w_Jw^{-1},q^{a}}}\dot{w}D(\theta)_J\\
&\ =\underline{U_{w_Jw^{-1},q^{a}}}\dot{w}D(\theta)_J+\sum_{t\in\mathbb{F}_{q^a}^*}
\varepsilon_{\alpha_k}(-t^{-1})\dot{s_k}h_k(t)\underline{U_{w_Jw^{-1},q^{a}}}\dot{w}D(\theta)_J\\
&\ =\underline{U_{w_Jw^{-1},q^{a}}}\dot{w}D(\theta)_J+\sum_{t\in\mathbb{F}_{q^a}^*}
\theta^w(h_k(t))\varepsilon_{\alpha_k}(-t^{-1})\dot{s_k}\underline{U_{w_Jw^{-1},q^{a}}}\dot{w}D(\theta)_J\\
&\ =\underline{U_{w_Jw^{-1},q^{a}}}\dot{w}D(\theta)_J+\sum_{t\in\mathbb{F}_{q^a}^*}\theta^w(h_k(-t^{-1}))\varepsilon_{\alpha_k}(t)\underline{(U_{w_Jw^{-1},q^{a}})^{s_k}}\dot{s_k}\dot{w}D(\theta)_J.
\endaligned
$$

\medskip
Since $({\bf U}_{w_Jw^{-1}})^{s_k}\neq{\bf U}_{w_Jw^{-1}}$  we have $s_k\Phi_{w_Jw^{-1}}^-\backslash\Phi_{w_Jw^{-1}}^-\neq\varnothing$.
Denote $\delta=\max\{\op{ht}(\alpha)\mid\alpha\in s_k\Phi_{w_Jw^{-1}}^-\backslash\Phi_{w_Jw^{-1}}^-\}$, and choose $\gamma\in s_k\Phi_{w_Jw^{-1}}^-\backslash\Phi_{w_Jw^{-1}}^-$ such that $\op{ht}(\gamma)=\delta$. Now we set  $$\Gamma=\{\alpha\in\Phi_{w_Jw^{-1}}^-\cap s_k\Phi_{w_Jw^{-1}}^-\mid\op{ht}(\alpha)\geq \delta\}.$$
Let $\Gamma'=\Phi_{w_Jw^{-1}}^-\backslash\Gamma$ and
$$s_k\Phi_{w_Jw^{-1}}^-\backslash(\Gamma\cup\{\gamma\})=\{\gamma_1,\cdots,\gamma_m\}~~ \text{with}~~ \op{ht}(\gamma_1)\geq\cdots\geq\op{ht}(\gamma_m).$$

For any $\beta\in \Phi_{w_Jw^{-1}}^-\backslash s_k\Phi_{w_Jw^{-1}}^-$, since
$$w_Jw^{-1}s_k(\beta)=w_Jw^{-1}(\beta)-
\langle \beta,\alpha_k^\vee\rangle w_Jw^{-1}(\alpha_k)\in\Phi^+,$$
this forces $\langle \beta,\alpha_k^\vee\rangle<0$, and hence $s_k(\beta)-\beta\in\mathbb{Z}_{\geq0}\Phi^+$. It follows that $$\op{ht}(\beta) \leq \op{ht}(\gamma)= \delta$$ for any $\beta\in \Phi_{w_Jw^{-1}}^-\backslash s_k\Phi_{w_Jw^{-1}}^-$. Therefore, $A=\Gamma$ and $B=\Gamma'$ satisfy the assumption in Lemma \ref{key}.
 Moreover it is clear that the set $U_{\Gamma,q^b}:=\prod_{\alpha\in\Gamma}U_{\alpha,q^b}$ is a normal subgroup of $U_{w_Jw^{-1}, q^b}$ and $U_{w_Jw^{-1}s_k, q^b}$ for any $b>0$.

\medskip
Now we set
$$\xi_1:=\underline{U_{w_Jw^{-1},q^{a}}}\dot{w}D(\theta)_J,~\xi_2:=\sum_{t\in\mathbb{F}_{q^a}^*}\theta^w(h_k(-t^{-1}))\varepsilon_{\alpha_k}(t)\underline{(U_{w_Jw^{-1},q^{a}})^{s_k}}\dot{s_k}\dot{w}D(\theta)_J.$$
Let $C$ be a complete set of representatives of left cosets of $U_{\Gamma,q^a}$ in $U_{\Gamma,q^{2a}}$.  Then
\begin{equation}\label{xinc}
\sum_{x\in C}x\xi_1=\underline{U_{\Gamma,q^{2a}}}\cdot\prod_{\alpha\in\Gamma'}\underline{U_{\alpha,q^a}}\dot{w}D(\theta)_J,
\end{equation}
where the product is taken with respect to a fixed order in $\Gamma'$.
Since for any $t\in\mathbb{F}_{q^a}^*$, the conjugation of $\varepsilon_{\alpha_k}(t)$ takes $C$ to another complete set of representatives of the left cosets of $U_{\Gamma,q^a}$ in $U_{\Gamma,q^{2a}}$ by the normality of $U_{\Gamma,q^{2a}}$, we have
\begin{equation}\label{xxi2}
\sum_{x\in C}x\xi_2=\sum_{t\in\mathbb{F}_{q^a}^*}\theta^w(h_k(-t^{-1}))\varepsilon_{\alpha_k}(t)\underline{U_{\Gamma,q^{2a}}}\cdot\underline{U_{\gamma,q^a}}\cdot\prod_{i=1}^m\underline{U_{\gamma_i,q^a}}\dot{s_k}\dot{w}D(\theta)_J
\end{equation}
by Lemma \ref{key}, where the product is taken with respect to the order $\gamma_1,\cdots,\gamma_m$.

Let $C'$  be a complete set of representatives of left cosets of $U_{\gamma,q^a}$ in $U_{\gamma,q^{2a}}$.
The sets  $A=\Gamma$ and $B=\Gamma'$ also satisfy the assumption in Lemma \ref{key} (ii).
Then we have
\begin{equation}\label{=0}
\sum_{y\in C'}y\sum_{x\in C}x\xi_1=q^a\sum_{x\in C}x\xi_1=0
\end{equation}
by (\ref{xinc}) and Lemma \ref{key} (ii). Since $[y^{-1},\varepsilon_{\alpha_k}(t)^{-1}]\in U_{\Gamma,q^{2a}}$ for any $y\in C'$ and $t\in\mathbb{F}_{q^a}^*$ (here $[a,b]=aba^{-1}b^{-1}$), we have
\begin{equation}\label{c'c}
\sum_{y\in C'}y\sum_{x\in C}x\dot{s_k}^{-1}v=\sum_{t\in\mathbb{F}_{q^a}^*}\theta^w(h_k(-t^{-1}))\varepsilon_{\alpha_k}(t)\underline{U_{\Gamma,q^{2a}}}\cdot\underline{U_{\gamma,q^{2a}}}\cdot\prod_{i=1}^m\underline{U_{\gamma_i,q^a}}\dot{s_k}\dot{w}D(\theta)_J\in V
\end{equation}
by (\ref{xxi2}), (\ref{=0}), and Lemma \ref{key} (i).

For any $i$ and $x\in C_i$, we have $$[x^{-1},\varepsilon_{\alpha_k}(t)^{-1}]\in \displaystyle\prod_{{r,s>0}\atop{r\gamma_i+s\alpha_k\in\Phi}}U_{r\gamma_i+s\alpha_k,q^{2a}}$$ by commutator relations in ${\bf U}$. Since (1) $\Phi_{w_Jw^{-1}s_k}^-=s_k\Phi_{w_Jw^{-1}}^-\cup\{\alpha_k\}$; (2) ${\rm ht}(r\gamma_i+s\alpha_k)>{\rm ht}(\gamma_i)$; (3) $r\gamma_i+s\alpha_k\ne\alpha_k$; (4) $r\gamma_i+s\alpha_k\in\Phi_{w_Jw^{-1}s_k}^-$, It follows that $r\gamma_i+s\alpha_k\in\Gamma\cup\{\gamma,\gamma_1,\cdots,\gamma_{i-1}\}$ (by our ht-downward arrangement), and hence $[x^{-1},\varepsilon_{\alpha_k}(t)^{-1}]\in U_{\Gamma,q^{2a}}U_{\gamma,q^{2a}}U_{\gamma_1,q^{2a}}\cdots U_{\gamma_{i-1},q^{2a}}$ (which is a group by the commutator formula). For each $1\leq i\leq m$, let $C_i$ be a complete set of representatives of the left cosets of $U_{\gamma_i,q^a}$ in $U_{\gamma_i,q^{2a}}$ and $\psi_i=\sum_{x\in C_i}x$.
Therefore, we have
$$
\aligned
\xi &\ :=\psi_m\psi_{m-1}\cdots\psi_1\sum_{y\in C'}y\sum_{x\in C}x\dot{s_k}^{-1}v\\ &\ =\sum_{t\in\mathbb{F}_{q^a}^*}\theta^w(h_k(-t^{-1}))\varepsilon_{\alpha_k}(t)\underline{(U_{w_Jw^{-1},q^{2a}})^{s_k}}\dot{s_k}\dot{w}D(\theta)_J\in V \endaligned
$$
by Lemma \ref{key} (i) and (\ref{s-1us}), and hence
$$
\aligned
\dot{s_k}^{-1}\xi &\ =\sum_{t\in\mathbb{F}_{q^a}^*}\theta^w(h_k(-t^{-1}))\varepsilon_{-\alpha_k}(-t)\underline{U_{w_Jw^{-1},q^{2a}}}\dot{w}D(\theta)_J\\
&\ =\sum_{t\in\mathbb{F}_{q^a}^*}\theta^w(h_k(-t^{-1}))\varepsilon_{\alpha_k}(-t^{-1})\dot{s_k}h_k(t)\underline{U_{w_Jw^{-1},q^{2a}}}\dot{w}D(\theta)_J\\
&\ =\theta^w(h_k(-1))\sum_{t\in\mathbb{F}_{q^a}^*}\varepsilon_{\alpha_k}(-t^{-1})\dot{s_k}\underline{U_{w_Jw^{-1},q^{2a}}}\dot{w}D(\theta)_J.
\endaligned
$$
It follows that
\begin{equation}\label{neq0}
\sum_{t\in\mathbb{F}_{q^a}^*}\varepsilon_{\alpha_k}(-t^{-1})\underline{(U_{w_Jw^{-1},q^{2a}})^{s_k}}\dot{s_k}\dot{w}D(\theta)_J\in V.
\end{equation}
Combining (\ref{any}) and (\ref{neq0}) yields $\underline{U_{w_Jw^{-1},q^{2a}}}\dot{w}D(\theta)_J\in V$ which completes the proof.
\end{proof}

\section{The natural characteristic (non-antidominant case)   }

In this section we consider the non-antidominant case. The main result of this section is

\begin{theorem}\label{infinite}
If $\theta$ is not antidominant, then the $\Bbbk {\bf G}$-module $\mathbb{M}(\theta)$ has an infinite submodule filtration. In particular, $\mathbb{M}(\theta)$ has infinite length.
\end{theorem}
Combining Theorem \ref{antidominant} and Theorem \ref{infinite}, we get the following theorem.
\begin{theorem}\label{defchar}
The  $\Bbbk {\bf G}$-module  $\mathbb{M}(\theta)$ has a composition series if and only of $\theta$ is antidominant, in which case $\mathbb{M}(\theta)$ has exactly $2^{|I(\theta)|}$ pairwise nonisomorphic composition factors, each occurring with multiplicity 1.
\end{theorem}

First we deal with the case when ${\bf G}=SL_2(\bar{\mathbb{F}}_q)$. From here to Corollary \ref{infsl2}, we assume that ${\bf G}=SL_2(\bar{\mathbb{F}}_q)$ and $\lambda\in\mathbb{Z}_{>0}$, and denote $V(\lambda), H^0(\lambda)$, and $L(\lambda)$ the corresponding Weyl module, costandard module, and simple module, respectively. By \cite[Corollary 7.5]{CPSV}, we have

\begin{theorem}\label{indiso}
If $M$ is a finite-dimensional rational ${\bf G}$-module such that all highest weights of the composition factors of $M$ are less than $q^a$, then $N$ is a ${\bf G}$-submodule of $M$ if and only if $N$ is a $G_{q^a}$-submodule of $M$.
\end{theorem}

\medskip
Now we set
$$u_a=\begin{pmatrix}
1 & a \\ 0 & 1
\end{pmatrix},\quad s=\begin{pmatrix}
0 & 1 \\ -1 & 0
\end{pmatrix},\quad h(t)=\begin{pmatrix}
t & 0 \\ 0 & t^{-1}
\end{pmatrix}.$$
Let $\tilde{q}$ be a power of $q$ such that $\tilde{q}>\lambda$. From \cite[Proposition 5.2 c)]{Jan1} we see that there is a basis $v_i(\tilde{q})$ $(0\leq i\leq \tilde{q}-1-\lambda)$ of $H^0(\tilde{q}-1-\lambda)$ such that
 \begin{equation}\label{acofu}
 u_av_i(\tilde{q})=\sum_{0\leq k\leq i}\left(i\atop k\right)a^{i-k}v_k(\tilde{q}),
 \end{equation}
 and
\begin{equation}\label{acofu2}
sv_i(\tilde{q})=(-1)^{\tilde{q}-1-\lambda-i}v_{\tilde{q}-1-\lambda-i}(\tilde{q}),~
h(t)v_i(\tilde{q})=t^{\tilde{q}-1-\lambda-2i}v_i(\tilde{q}).
\end{equation}

\begin{lemma}\label{inclusion}
For any $r>0$ and $\lambda<\tilde{q}$, let $K_{\tilde{q}^r}$ be the kernel of the natural map $\op{Ind}_{B_{\tilde{q}^r}}^{G_{\tilde{q}^r}}\lambda\rightarrow V(\lambda)$. Then we have

\noindent$\op{(1)}$ $K_{\tilde{q}^r}$ is isomorphic to $H^0(\tilde{q}^r-1-\lambda)$.

\noindent$\op{(2)}$ There is an injective $G_{\tilde{q}}$-module homomorphism $K_{\tilde{q}}\hookrightarrow K_{\tilde{q}^r}$.
\end{lemma}
\begin{proof}
We first give the inclusion $H^0(\tilde{q}^r-1-\lambda)\hookrightarrow\op{Ind}_{B_{\tilde{q}^r}}^{G_{\tilde{q}^r}}\lambda$ explicitly.
 Consider the elements $\sum\limits_{t\in \mathbb{F}_{\tilde{q}^r}}t^iu_ts{\bf 1}_\lambda$ $(0\leq i< \tilde{q}^r-1-\lambda)$ and $(-1)^\lambda{\bf 1}_\lambda+\sum\limits_{t\in \mathbb{F}_{\tilde{q}^r}}t^{\tilde{q}^r-1-\lambda}u_ts{\bf 1}_\lambda$. A direct calculation shows that for any $0\leq i\leq \tilde{q}^r-1-\lambda$, we have
\begin{equation}\label{eq1}
u_a(-1)^i\sum_{t\in \mathbb{F}_{\tilde{q}^r}}t^iu_ts{\bf 1}_\lambda=\sum_{0\leq k\leq i}\left(i\atop k\right)a^{i-k}(-1)^{k}\sum_{t\in \mathbb{F}_{\tilde{q}^r}}t^ku_ts{\bf 1}_\lambda ~(a\in\mathbb{F}_{\tilde{q}^r});
$$
$$h(b)\sum_{t\in\mathbb{F}_{\tilde{q}^r}}t^iu_ts{\bf 1}_\lambda=b^{\tilde{q}^r-1-\lambda-2i}\sum_{t\in\mathbb{F}_{\tilde{q}^r}}t^iu_ts{\bf 1}_\lambda~(b\in\mathbb{F}_{\tilde{q}^r}^*).
\end{equation}
and
\begin{equation}\label{eq2}
s\sum\limits_{t\in \mathbb{F}_{\tilde{q}^r}}t^iu_ts{\bf 1}_\lambda=\left\{\begin{array}{ll}
(-1)^i\sum\limits_{t\in \mathbb{F}_{\tilde{q}^r}}t^{\tilde{q}^r-1-\lambda-i}u_ts{\bf 1}_\lambda &\  0<i\leq \tilde{q}^r-1-\lambda\\
(-1)^\lambda{\bf 1}_\lambda+\sum\limits_{t\in \mathbb{F}_{\tilde{q}^r}}t^{\tilde{q}^r-1-\lambda}u_ts{\bf 1}_\lambda &\  i=0.
\end{array}\right.
\end{equation}

By (\ref{acofu}), (\ref{acofu2}), (\ref{eq1}), and (\ref{eq2}), we see that the map
\begin{equation}\label{eq}
v_k(\tilde{q}^r)\mapsto\left\{\begin{array}{ll}
(-1)^k\sum\limits_{t\in \mathbb{F}_{\tilde{q}^r}}t^ku_ts{\bf 1}_\lambda &\  0\leq k<\tilde{q}^r-1-\lambda\\
(-1)^{\tilde{q}^r-1-\lambda}\left((-1)^\lambda{\bf 1}_\lambda+\sum\limits_{t\in \mathbb{F}_{\tilde{q}^r}}t^{\tilde{q}^r-1-\lambda}u_ts{\bf 1}_\lambda\right) &\ k=\tilde{q}^r-1-\lambda
\end{array}\right.
\end{equation}
gives the desired inclusion.

Routine calculations show that the  elements in the right hand side of (\ref{eq}) is in $K_{\tilde{q}^r}$, and hence the comparison of the dimension gives the exact sequence
\begin{equation}\label{finexseq}
0\rightarrow H^0(\tilde{q}^r-1-\lambda)\rightarrow\op{Ind}_{B_{\tilde{q}^r}}^{G_{\tilde{q}^r}}\lambda\rightarrow V(\lambda)\rightarrow0
\end{equation}
which proves (1). Statement (2) follows immediately from (\ref{finexseq}).
\end{proof}

By Lemma \ref{inclusion}, if $a|b$, then there is an injective $G_{q^a}$-module homomorphism $H^0(q^a-1-\lambda)\rightarrow H^0(q^b-1-\lambda)$. This family of injection forms a direct system so that one can form the direct limit $H_{\lambda}^0=\bigcup\limits_{a>0}H^0(q^a-1-\lambda)$ (which is a $\Bbbk{\bf G}$-module). Taking direct limits as $a\rightarrow\infty$ to (\ref{finexseq})
yields the exact sequence
\begin{equation}\label{exseq}
0\rightarrow H_{\lambda}^0\rightarrow\mathbb{M}(\lambda)\rightarrow V(\lambda)\rightarrow0.
\end{equation}

By Lemma \ref{inclusion}, one can for $0\le i\le \widetilde{q}-1-\lambda$ regard $v_i(\tilde{q})$ as an element of $H^0(\tilde{q}^r-1-\lambda)$. With this identification, we have
\begin{lemma}\label{extend}
For any $0\leq i\leq \tilde{q}-1-\lambda$ and $r>1$, we have
$$v_i(\tilde{q})=\sum_{k=0}^{\tilde{q}+\cdots+\tilde{q}^{r-1}}v_{i+k(\tilde{q}-1)}(\tilde{q}^r)$$
in $H^0(\tilde{q}^r-1-\lambda)$. In particular, this gives the injection in Lemma \ref{inclusion} (2) explicitly.
\end{lemma}
\begin{proof}
Notice that for $t\in\mathbb{F}_{\tilde{q}^r}$
$$
\sum_{k=0}^{\tilde{q}+\cdots+\tilde{q}^{r-1}}t^{k(\tilde{q}-1)}=\left\{\begin{array}{cl}
(t^{\tilde{q}^r-1}-1)/(t^{\tilde{q}-1}-1)=0 &\  t\not\in\mathbb{F}_{\tilde{q}}\\
1 &\ t\in\mathbb{F}_{\tilde{q}}.
\end{array}\right.
$$
If $i<\tilde{q}-1-\lambda$, then $i+k(\tilde{q}-1)<\tilde{q}^q-1-\lambda$ for any $0\le k\le\tilde{q}+\cdots\tilde{q}^{r-1}$. It follows from (\ref{eq}) that
$$
\aligned
\sum_{k=0}^{\tilde{q}+\cdots+\tilde{q}^{r-1}}v_{i+k(\tilde{q}-1)}(\tilde{q}^r) &\ =\sum_{t\in\mathbb{F}_{\tilde{q}^r}}(-1)^i\sum_{k=0}^{\tilde{q}+\cdots+\tilde{q}^{r-1}}t^{i+k(\tilde{q}-1)}u_ts{\bf 1}_\lambda\\
&\ =(-1)^i\sum_{t\in\mathbb{F}_{\tilde{q}}}t^iu_ts{\bf 1}_\lambda\\
&\ =v_i(\tilde{q}),
\endaligned
$$
and
$$
\aligned
\sum_{k=0}^{\tilde{q}+\cdots+\tilde{q}^{r-1}}v_{\tilde{q}-1-\lambda+k(\tilde{q}-1)}(\tilde{q}^r) &\ =\sum_{t\in\mathbb{F}_{\tilde{q}^r}}(-t)^{\tilde{q}-1-\lambda}\sum_{k=0}^{\tilde{q}+\cdots+\tilde{q}^{r-1}}(-t)^{k(\tilde{q}-1)}u_ts{\bf 1}_\lambda\\
&\ ~~~+(-1)^\lambda(-1)^{\tilde{q}^r-1-\lambda}{\bf 1}_\lambda\\
&\ =(-1)^{\tilde{q}^r-1-\lambda}\left(\sum_{t\in\mathbb{F}_{\tilde{q}}}t^{\tilde{q}-1-\lambda}u_ts{\bf 1}_\lambda+(-1)^\lambda{\bf 1}_\lambda\right)\\
&\ =v_{\tilde{q}-1-\lambda}(\tilde{q}),
\endaligned
$$
which completes the proof.
\end{proof}

\begin{lemma}\label{wtvec}
Let $r>0$ and $V$ be a $G_{\tilde{q}^r}$-submodule of $H^0(\tilde{q}^r-1-\lambda)$. If
$v=\sum_ic_iv_i(\tilde{q}^r)\in V$, then $c_j\neq0\Rightarrow v_j(\tilde{q}^r)\in V$.
\end{lemma}
\begin{proof}
Clearly, the result holds when $G_{\tilde{q}^r}$ is replaced by ${\bf G}$. Since all highest weights of the composition factors of $H^0(\tilde{q}^r-1-\lambda)$ are $\tilde{q}^r$-restricted, the result follows immediately from Theorem \ref{indiso}.
\end{proof}

Since the proof of Theorem \ref{infinite} is technical, we give the following main idea of the proof first for the convenience of readers.

\smallskip
\noindent{\bf Main idea}: The general case can be reduced to ${\bf G}=SL_2(\bar{\mathbb{F}}_q)$ (see the discussion after Corollary \ref{infsl2} ).
From (\ref{exseq}), it is enough to show that $H_{\lambda}^0$ has infinite length. To see this, it suffices to prove that $\Bbbk{\bf G}v_0(q^{a^i})$ $(i\in\mathbb{N})$ is a strictly descending chain of $\Bbbk{\bf G}$-submodules of $H_\lambda^0$ if $a$ is large enough (see (\ref{infchain}) below). Thus, it is enough to show that $\Bbbk{\bf G}v_0(q^b)\subsetneq \Bbbk{\bf G}v_0(q^a)$ if $a$ is large enough and $a|b$. We prove this by showing that $\Bbbk G_{q^s}v_0(q^b)\subsetneq\Bbbk G_{q^s}v_0(q^a)$ holds in $H^0(q^s-1-\lambda)$ for any common multiple $s$ of $a,b$. Since the submodule structure of $H^0(q^s-1-\lambda)$ is known (Proposition \ref{De} below), one can use this to estimate ``how large" is $\Bbbk G_{q^s}v_0(q^b)$ and $\Bbbk G_{q^s}v_0(q^a)$ in $H^0(q^s-1-\lambda)$ (Lemma \ref{rs} below), and use these information to conclude that $\Bbbk G_{q^s}v_0(q^b)\subsetneq\Bbbk G_{q^s}v_0(q^a)$ (see the proof of Lemma \ref{bat}) which completes the proof.

\bigskip
Now we turn to technical details. Let $m\in\mathbb{Z}_{>0}$. Following \cite{De}, for each positive integer $j$ we define $\rho_j(m)=m-2r_j$, where $m+1=\lambda_jp^j+r_j$ with $\lambda_j\geq0$ and $0\leq r_j<p^j$. We call $\rho_j$ an {\it $m$-admissible} reflection if $p\nmid\lambda_j$. A strictly decreasing sequence of positive integers $$m,~\rho_{e_k}(m),~\rho_{e_{k-1}}\rho_{e_k}(m),~\cdots,\rho_{e_1}\rho_{e_2}\cdots\rho_{e_k}(m)$$
is called  {\it $m$-admissible} if the following conditions are satisfied:

\noindent(a) $0<e_1<e_2<\cdots<e_k$;

\noindent(b) For each $0\leq j\leq k$, $\rho_{e_j}$ is $\rho_{e_{j+1}}\rho_{e_{j+2}}\cdots\rho_{e_k}(m)$-admissible.

\noindent For convenience, in the following we simply call $\boldsymbol{e}=(e_1,\cdots,e_k)$ $m$-admissible (by abuse of terminology) if these conditions are satisfied, and set $\rho_{\boldsymbol{e}}=\rho_{e_1}\cdots\rho_{e_k}$.

Let $\mathbb{S}(m)=\{\rho_{\boldsymbol{e}}(m)|\boldsymbol{e}~\mbox{is}~m\op{-admissible}\}$. Following \cite{De}, there is a partial order $\preccurlyeq$ on $\mathbb{S}(m)$ defined as follows:
for $\lambda_1,\lambda_2\in\mathbb{S}(m)$, write $\frac{m-\lambda_2}{2}=\sum_tm_tp^t$ and $\frac{m-\lambda_1}{2}=\sum_tn_tp^t$ for their $p$-adic expansion.  Set $\lambda_1\succcurlyeq\lambda_2$ if $m_t=0$ implies $n_t=0$ for all $t$. The following proposition is the dual version of the main theorem in \cite{De}.

\begin{Prop}[{\cite[pp.251 Section 2]{De}}]\label{De}
Let $v_i$ be the weight vector of $m-2i$ in $H^0(m)$. Then every submodule of $H^0(m)$ has the form
$$L_{E}=\sum_{\mu\in E}\Bbbk{\bf G}v_{(m-\mu)/2}$$
for some subset $E$ of $\mathbb{S}(m)$. Moreover, if $\nu\in\mathbb{S}(m)$, then $L(\nu)$ is a composition factor of $L_E$ if and only if $\nu\succcurlyeq\mu$ for some $\mu\in E$.
\end{Prop}
\noindent In particular, we have
\begin{Cor}\label{cfsl2}
The irreducible module $L(m')$ is a composition factor of $H^0(m)$ if and only if $m'=\rho_{\boldsymbol{e}}(m)$ for some $m$-admissible sequence $\boldsymbol{e}$.
\end{Cor}

\noindent From here to the end of this section, we write $q=p^d$. For $r$ with $q^r>\lambda$, set $\mu_r=q^r-1-\lambda$. Let $l$ be the the largest number such that the coefficient in the $p$-adic expansion of $\lambda$ of $p^l$ is nonzero.
\begin{lemma}\label{lambdarho}
Let $\boldsymbol{e}$ be a $\lambda-1$-admissible sequence. If $h$ is an integer which satisfies $l< h < rd$, then there is an integer $0<\lambda_{\boldsymbol{e}}\le\lambda$ which is independent of $h$ such that
$(\mu_r-\rho_{\boldsymbol{e}}\rho_h(\mu_r))/2=p^h-\lambda_{\boldsymbol{e}}$.
In particular, if $h<h'$, then
$\rho_{\boldsymbol{e}}\rho_h(\mu_r)\succcurlyeq\rho_{\boldsymbol{e}}\rho_{h'}(\mu_r)$.
\end{lemma}
\begin{proof}
Direct calculation shows that
\begin{equation}\label{lambda-1}
\rho_h(\mu_r)=Q(r,h)+\lambda-1,
\end{equation}
where $Q(r,h)=(p-1)(p^{rd-1}+\cdots+p^{h+1})+(p-2)p^h$, and
$(\mu_r-\rho_h(\mu_r))/2=p^h-\lambda$.
 Since $\lambda<p^h$ and $p^h|Q(r,h)$, we have
 $\rho_{\boldsymbol{e}}\rho_h(\mu_r)=Q(r,h)+\rho_{\boldsymbol{e}}(\lambda-1)$,
 and hence
$$(\mu_r-\rho_{\boldsymbol{e}}\rho_h(\mu_r))/2=(\mu_r-\rho_h(\mu_r))/2+(\lambda-1-\rho_{\boldsymbol{e}}(\lambda-1))/2.$$
We take
$\lambda_{\boldsymbol{e}}=\lambda-(\lambda-1-\rho_{\boldsymbol{e}}(\lambda-1))/2$
as desired.
\end{proof}

\noindent Now Corollary \ref{cfsl2} becomes

\begin{lemma}\label{cfsl22}
The composition factors of $H^0(\mu_r)$ are all $L(\rho_{\boldsymbol{f}}(\mu_r))$ and $L(\rho_{\boldsymbol{e}}\rho_h(\mu_r))$, where
$\boldsymbol{f}=(f_1,\cdots,f_k)$ is $\mu_r$-admissible and $f_k\leq l$, and $\boldsymbol{e}$ is $\lambda-1$-admissible and $l< h < rd$.
\end{lemma}
\begin{proof}
 Let $L(\lambda)$ be a composition factor of $H^0(\mu_r)$. Then $\lambda=\rho_{n_1}\cdots\rho_{n_k}(\mu_r)$, where $(n_1,\cdots,n_k)$ is $\mu_r$-admissible. These already contain $L(\rho_{\boldsymbol{f}}(\mu_r))$ with $\boldsymbol{f}=(f_1,\cdots,f_k)$ is $\mu_r$-admissible and $f_k\leq l$. (Recall: $l$ be the the largest number such that the coefficient in the $p$-adic expansion of $\lambda$ of $p^l$ is nonzero). Consider the remaining factors (the case $h:=n_k>l$). We claim that $n_{k-1}\le l$. In fact, suppose $n_{k-1}>l$. By the proof of Lemma 5.10, $\rho_h(\mu_r)+1=Q(r,h)+\lambda$, where $p^h|Q(r,h)$. Thus $\rho_h(\mu_r)+1=\lambda_hp^{n_{k-1}}+\lambda$, with $p^{h-n_{k-1}}|\lambda_h$ (implies $p|\lambda_h$). This contradicts to ``$\rho_{n_{k-1}}$ is $\rho_h(\mu_r)$-admissible". Therefore $n_1<\cdots<n_{k-1}\le l$ and $(n_1,\cdots,n_{k-1})$ is $\lambda-1$-admissible since $\rho_{n_1}\cdots\rho_{n_{k-1}}\rho_h(\mu_r)=Q(r,h)+\rho_{n_1}\cdots\rho_{n_{k-1}}(\lambda-1)$.
\end{proof}

From here to the end of the proof of Lemma \ref{rs}, we choose $r\in\mathbb{Z}_{>0}$ so that $q^r>\lambda$ and $s=rt$ with $t>1$. For each $\mu_s$-admissible sequence $\boldsymbol{f}=(f_1,\cdots,f_k)$ with $f_k\leq l$, set
$j_{\boldsymbol{f}}=(\mu_s-\rho_{\boldsymbol{f}}(\mu_s))/2$.

\begin{lemma}\label{rs}
Let $\boldsymbol{e}$ be a $\lambda-1$-admissible sequence, and $i_{\boldsymbol{e}}$ be the number such that $p^{i_{\boldsymbol{e}}}|\lambda_{\boldsymbol{e}}$ and $p^{i_{\boldsymbol{e}}+1}\nmid\lambda_{\boldsymbol{e}}$. Then
$$v_{p^{(t-1)rd+i_{\boldsymbol{e}}}-\lambda_{\boldsymbol{e}}}(q^s)\in\Bbbk G_{q^s}v_0(q^r)$$
for any $\lambda-1$-admissible sequence $\boldsymbol{e}$, and
$$\Bbbk G_{q^s}v_0(q^r)\subset M:=\sum_{\boldsymbol{e}}\Bbbk G_{q^s}v_{p^{(t-1)rd+i_{\boldsymbol{e}}}-\lambda_{\boldsymbol{e}}}(q^s)
+\sum_{\boldsymbol{f}}\Bbbk G_{q^s}v_{j_{\boldsymbol{f}}}(q^s),$$
where $\boldsymbol{e}$ runs over all $\lambda-1$--admissible sequences, and $\boldsymbol{f}$ runs over all $\mu_s$-admissible sequences $(f_1,\cdots,f_k)$ with $f_k\leq l$.
\end{lemma}
\begin{proof}
Since the coefficient of $p^i$ in $p$-adic expansion of $p^{(t-1)rd+i_{\boldsymbol{e}}}-\lambda_{\boldsymbol{e}}$ is zero if $i>(t-1)rd+i_{\boldsymbol{e}}-1$ or $i<i_{\boldsymbol{e}}$ by assumption, we have
{\small
$$\left({p^{i_{\boldsymbol{e}}}(1+q^r+\cdots+q^{(t-2)r})(q^r-1)}\atop{p^{(t-1)rd+i_{\boldsymbol{e}}}-\lambda_{\boldsymbol{e}}}\right)
=\left({(p-1)(p^{i_{\boldsymbol{e}}}+
\cdots+p^{(t-1)rd+i_{\boldsymbol{e}}-1})}\atop{p^{(t-1)rd+i_{\boldsymbol{e}}}-
\lambda_{\boldsymbol{e}}}\right)\neq0\quad(\op{mod}p)$$}
according to Lemma \ref{Binomial}.
Combining this and $$p^{i_{\boldsymbol{e}}}(1+q^r+\cdots+q^{(t-2)r})<q^r+\cdots+q^{(t-1)r}$$
yields $v_{p^{(t-1)rd+i_{\boldsymbol{e}}}-\lambda_{\boldsymbol{e}}}(q^s)\in\Bbbk G_{q^s}v_0(q^r)$ by Lemma \ref{extend} and \ref{wtvec}.

To show the second statement, we have to show that $v_{k(q^r-1)}(q^s)\in M$ for all $k\leq q^r+\cdots+q^{(t-1)r}$ by Lemma \ref{extend} and \ref{wtvec}. Suppose that $v_{k(q^r-1)}(q^s)\not\in M$ for some $k\leq q^r+\cdots+q^{(t-1)r}$. Then $v_{k(q^r-1)}(q^s)$ is a weight vector of some composition factor $L(\nu)$ of $H^0(\mu_s)/M$. It follows that $\left({k(q^r-1)}\atop{(\mu_s-\nu)/2}\right)\neq0~(\op{mod}p)$ by Lemma \ref{wtvec} and (\ref{acofu}). Combining Proposition \ref{De}, Lemma \ref{lambdarho} and \ref{cfsl22} yields $\frac{\mu_s-\nu}{2}=p^h-\lambda_{\boldsymbol{e}}$ for some $\lambda-1$-admissible sequence $\boldsymbol{e}$ and $h>(t-1)rd+i_{\boldsymbol{e}}$. Denote $h=(t-1)rd+i$, where $i>i_{\boldsymbol{e}}$. Thus, we have
\begin{equation}\label{suppose}
\left({k(q^r-1)}\atop{p^{(t-1)rd+i}-\lambda_{\boldsymbol{e}}}\right)\neq0\quad(\op{mod}p)
\end{equation}
by Lemma \ref{Binomial}.
From the $p$-adic expansion of $p^{(t-1)rd+i}-\lambda_{\boldsymbol{e}}$ and Lemma \ref{Binomial} we see that
\begin{equation}\label{sumsumsum}
k(q^r-1)=\sum_{i\leq j\leq rd-1}c_jp^{(t-1)rd+j}+(p-1)\sum_{l+1\leq j\leq(t-1)rd+i-1}p^j+\sum_{0\leq j\leq l}c_j'p^j
\end{equation}
for some $0\leq c_j,c_j'<p$ and $c_{i_{\boldsymbol{e}}}'\neq0$, where $l$ is the number introduced in Lemma \ref{lambdarho}. Denote RHS the right hand side of (\ref{sumsumsum}). We have
$$
\aligned
\op{RHS}&\ =\sum_{i\leq j\leq rd-1}c_jp^{(t-1)rd+j}+p^{(t-1)rd+i}-p^{l+1}+\sum_{0\leq j\leq l}c_j'p^j\\
&\ \equiv \sum_{i\leq j\leq rd-1}c_jp^j+p^i-p^{l+1}+\sum_{0\leq j\leq l}c_j'p^j\quad(\op{mod}q^r-1)
\endaligned
$$
Since $c_{i_{\boldsymbol{e}}}'\neq0$ and $i>i_{\boldsymbol{e}}$, we have
\begin{equation}\label{RHS}
1-q^r<\sum_{i\leq j\leq rd-1}c_jp^j+p^i-p^{l+1}+\sum_{0\leq j\leq l}c_j'p^j<q^r
\end{equation}
and
\begin{equation}\label{RHS2}
\sum_{i\leq j\leq rd-1}c_jp^j+p^i-p^{l+1}+\sum_{0\leq j\leq l}c_j'p^j\neq0.
\end{equation}
Combining the fact that $q^r-1|\op{RHS}$ with (\ref{suppose}), (\ref{RHS}), and (\ref{RHS2}), we see that all $c_j,c_j'$ are $p-1$, in which case
$$k(q^r-1)=(p-1)\sum_{0\leq j\leq rtd-1}p^j=q^{rt}-1,$$
and hence $k=1+q^r+\cdots+q^{(t-1)r}$ which contradicts to the assumption. This completes the proof.
\end{proof}

\begin{lemma}\label{bat}
If $b=at$, $t>1$, and $q^a>\lambda$, then $\Bbbk{\bf G}v_0(q^b)\subsetneq \Bbbk{\bf G}v_0(q^a)$.
\end{lemma}
\begin{proof}
We have to prove that if $b|s$, then $\Bbbk G_{q^s}v_0(q^b)\subsetneq\Bbbk G_{q^s}v_0(q^a)$ (Since this implies $v_0(q^a)\not\in\Bbbk G_{q^s}v_0(q^b)$ for any multiple $s$ of $b$, and hence $v_0(q^a)\not\in\Bbbk{\bf G}v_0(q^b)$). Write $s=bt'$ and let $\boldsymbol{e}$ be a $\lambda-1$-admissible sequence. By Lemma \ref{rs} we have
\begin{equation}\label{bigabi1}
v_{p^{(tt'-1)ad+i_{\boldsymbol{e}}}-\lambda_{\boldsymbol{e}}}(q^s)\in\Bbbk G_{q^s}v_0(q^a).
\end{equation}
On the other hand, since $\rho_{\boldsymbol{e}'}\rho_{(t'-1)bd+i_{\boldsymbol{e}'}}(\mu_s)\not\preccurlyeq\rho_{\boldsymbol{e}}\rho_{(tt'-1)ad+i_{\boldsymbol{e}}}(\mu_s)$ for any $\lambda-1$-admissible sequence $\boldsymbol{e}'$ (this follows from $(t'-1)bd+i_{\boldsymbol{e}'}<(tt'-1)ad+i_{\boldsymbol{e}}$ since $q^a>\lambda$ and looking at the $p$-adic expansion of both sides), and $\rho_{\boldsymbol{f}}(\mu_s)\not\preccurlyeq \rho_{\boldsymbol{e}}\rho_{(tt'-1)ad+i_{\boldsymbol{e}}}(\mu_s)$ (follows immediately from definition) for any $\mu_s$-admissible sequence $\boldsymbol{f}=(f_1,\cdots,f_k)$ with $f_k\leq l$, we have
\begin{equation}\label{bigabi2}
v_{p^{(tt'-1)ad+i_{\boldsymbol{e}}}-\lambda_{\boldsymbol{e}}}(q^s)\not\in\Bbbk G_{q^s}v_0(q^b)
\end{equation}
by Lemma \ref{De} and \ref{rs}. Combining (\ref{bigabi1}) and (\ref{bigabi2}) yields $\Bbbk G_{q^s}v_0(q^b)\subsetneq\Bbbk G_{q^s}v_0(q^a)$ which completes the proof.
\end{proof}

\begin{Cor}\label{infsl2}
Let ${\bf G}=SL_2(\bar{\mathbb{F}}_q)$ and $\lambda\in\mathbb{Z}_{>0}$. Then the $\Bbbk {\bf G}$-module $\mathbb{M}(\lambda)$ has an infinite submodule filtration.
\end{Cor}
\begin{proof}
Choose $a$ so that $q^a>\lambda$. By (\ref{exseq}), it is sufficient to find an infinite (proper) submodule filtration for $H_\lambda^0$. But this is done by taking the chain
\begin{equation}\label{infchain}
H_\lambda^0\supset\Bbbk{\bf G}v_0(q^a)\supsetneq\Bbbk{\bf G}v_0(q^{a^2})\supsetneq\cdots\supsetneq\Bbbk{\bf G}v_0(q^{a^i})\supsetneq\cdots
\end{equation}
thanks to Lemma \ref{bat}.
\end{proof}

Now we return to the general case. For any $\alpha_i\in\Delta$, let ${\bf P}_i={\bf B}\cup{\bf B}\dot{s_i}{\bf B}$ be the corresponding parabolic subgroup and ${\bf L}_i$ the Levi subgroup of ${\bf P}_i$. Let ${\bf U}_i$ be the unipotent radical of ${\bf P}_i$. Then ${\bf P}_i={\bf L}_i\ltimes {\bf U}_i$. Moreover, ${\bf B}_i={\bf B}\cap{\bf L}_i$ is a Borel subgroup of ${\bf L}_i$. By abusing of notation, we also denote by $\Bbbk_\theta$ for its restriction to ${\bf B}_i$, set $\mathbb{M}_i(\theta)=\Bbbk{\bf L}_i\otimes_{\Bbbk {\bf B}_i}\Bbbk_\theta$. Let ${\bf U}_i$ acts on $\mathbb{M}_i(\theta)$ trivially. Then $\mathbb{M}_i(\theta)$ becomes a ${\bf P}_i$-module.  The following result is an easy consequence of Lemma 2.2 in \cite{Xi}.

\begin{lemma}\label{trans}
The $\Bbbk {\bf G}$-module $\mathbb{M}(\theta)$ is isomorphic to $\Bbbk{\bf G}\otimes_{\Bbbk {\bf P}_i}\mathbb{M}_i(\theta)$.
\end{lemma}

Let $G$ and $G'$ be connected reductive algebraic groups. An {\it isogeny} $\pi: G\rightarrow G'$ of algebraic groups is a surjective rational homomorphism with finite kernel. Such an isogeny $\pi$ is called a {\it central isogeny} if $\pi$ induces an isomorphism in the sense of algebraic groups of each root subgroup of $G$ onto its image. For details on isogenies, one can refer \cite{Sp}.

\begin{proof}[\it Proof of Theorem \ref{infinite}]
By assumption, $\langle\theta,\alpha_i^\vee\rangle>0$ for some $\alpha_i\in\Delta$. Denote ${\bf P}_i$, ${\bf L}_i$, ${\bf B}_i$, $\mathbb{M}_i(\theta)$ as above. Since ${\bf L}_i$ is a reductive group of rank 1, there is a central isogeny $\pi:~SL_2(\bar{\mathbb{F}}_q)\times T'\rightarrow{\bf L}_i$ which maps $B\times T'$ to ${\bf B}_i$, and $T\times T'$ to ${\bf T}$ (here $T'$ is a torus, and $B$, $T$ are the standard Borel subgroup and maximal torus of $SL_2(\bar{\mathbb{F}}_q)$, respectively). As $SL_2(\bar{\mathbb{F}}_q)\times T'$-modules, we see that
$$\mathbb{M}_i(\theta)  \cong \left(\op{Ind}_B^{SL_2(\bar{\mathbb{F}}_q)}\Bbbk_{\theta_i}\right)\otimes \Bbbk_{\theta'},$$
where $\Bbbk_{\theta_i}$ is the restriction of $\Bbbk_{\theta}$ to $T$ via $\pi$ and $\Bbbk_{\theta'}$ is the restriction of $\Bbbk_{\theta}$ to $T'$ via $\pi$.
Therefore, $\mathbb{M}_i(\theta)$ has infinite many composition factors by Corollary \ref{infsl2},  Since the functor $\Bbbk{\bf G}\otimes_{\Bbbk {\bf P}_i}-$ is exact ($\Bbbk{\bf G}$ is free over $\Bbbk{\bf P}_i$), it follows from Corollary \ref{infsl2} and Lemma \ref{trans} that $\mathbb{M}(\theta)$ has an infinite submodule filtration, i.e, it has no composition series (of finite length).
\end{proof}

\section{Conclusion and consequences}

Let $p$ be a prime number, and let $q$ be a power of $p$. Let $\theta\in \widehat{\bf T}$ be a character. To summarize, combining Theorem \ref{main}, \ref{antidominant}, \ref{infinite} yields the following
\begin{theorem}\label{conclusion}
If $\op{char}\Bbbk\neq p$, then all $\Bbbk {\bf G}$-modules $E(\theta)_J$ with $J\subset I(\theta)$ are irreducible and pairwise nonisomorphic. In particular, the  $\Bbbk {\bf G}$-module $\mathbb{M}(\theta)$ has $2^{|I(\theta)|}$ composition factors. If $\Bbbk=\bar{\mathbb{F}}_q$ and $\theta\in X({\bf T})$, then $\mathbb{M}(\theta)$ has a composition series of finite length if and only if $\theta$ is antidominant.
\end{theorem}

Let $\op{tr}$ be the trivial ${\bf B}$-module. We denote $E(\op{tr})_J$ simply by $E_J$.  Since $\mathbb{M}(\op{tr})$ is realizable over any field, as a consequence of Theorem \ref{conclusion} we have
the following results which were first proved in \cite{CD1} and \cite{CD2}.
\begin{Cor}
Let $\Bbbk$ be any field. Then all  $\Bbbk {\bf G}$-modules $E_J$ with $J\subset I$ are irreducible and pairwisely nonisomorphic. In particular $\mathbb{M}(\op{tr})$ has $2^{|I|}$ composition factors.
\end{Cor}

Assume that $\Bbbk=\bar{\mathbb{F}}_q$ and $\theta\in X({\bf T})$. We call $\theta$ {\it strongly antidominant} if $\langle\theta,\alpha^\vee\rangle<0$ for any $\alpha\in\Delta$. As another consequence of Theorem \ref{conclusion}, we get a necessary and sufficient condition for the irreducibility of $\mathbb{M}(\theta)$.

\begin{Cor}\label{ifandonlyif}
Let $\Bbbk=\bar{\mathbb{F}}_q$ and $\theta\in X({\bf T})$. Then the $\Bbbk {\bf G}$-module $\mathbb{M}(\theta)$ is irreducible if and only if $\theta$ is strongly antidominant.
\end{Cor}

Assume that $\Bbbk=\bar{\mathbb{F}}_q$ and $\theta\in X({\bf T})$ is strongly antidominant. Corollary \ref{ifandonlyif} suggests that the property of $\mathbb{M}(\theta)$ is analogous to the antidominant (in the sense of ``dot" action) Verma modules in the category $\mathcal{O}$ of complex semisimple Lie algebras (see Chapter 4 and Chapter 5 in \cite{Hum}).

\begin{Cor}\label{54}
Let $V$ be an $($abstract$)$ irreducible $\Bbbk{\bf G}$-module. Assume that one of the following statements holds:

\smallskip
\noindent$\op{(i)}$ $\op{char}\Bbbk\neq p$ and $V$ contains a $\bf B$-stable line;

\noindent$\op{(ii)}$  $\Bbbk=\bar{\mathbb{F}}_q$ and $V$ contains a ${\bf B}$-stable line on which ${\bf T}$ acts by some antidominant weight $\theta\in\widehat{\bf T}$.

\smallskip
\noindent Then $V$ is isomorphic to $E(\theta)_\varnothing$ for some $\theta\in\widehat{\bf T}$.
\end{Cor}
\begin{proof}
Let $\Bbbk v$ be the $\bf B$-stable line in $V$. Then ${\bf U}$ acts on $v$ trivially and ${\bf T}$-acts on $v$ by some character $\theta\in\widehat{\bf T}$. Since $V=\Bbbk{\bf G}v$, $V$ is an irreducible quotient $\mathbb{M}(\theta)$.
By Theorem \ref{main} and \ref{defchar}, we have $V=E(\theta)_J$ for some $J\subset I(\theta)$.
Therefore $E(\theta)_J^{\bf U}\neq0$, which forces $J=\varnothing$.
This completes the proof.
\end{proof}

\noindent It follows immediately from Corollary \ref{54} that
\begin{Cor}
Assume either $(1)$ $\op{char}\Bbbk\neq p$ or $(2)$ $\Bbbk=\bar{\mathbb{F}}_q$ and $\theta$ is antidominant, then
the $\Bbbk {\bf G}$-module $\mathbb{M}(\theta)$ has simple head which is isomorphic to $E(\theta)_\varnothing$.
\end{Cor}

\begin{Cor}
Assume $\Bbbk$ is algebraically closed with $\op{char}\Bbbk\neq\op{char}\bar{\mathbb{F}}_q$.
Then any finite-dimensional irreducible representation of $\bf G$ is one-dimensional and isomorphic to $E(\theta)_\varnothing$ for some $\theta\in\widehat{\bf T}$ with $I(\theta)=I$.
\end{Cor}
\begin{proof} Let $V$ be a finite-dimensional irreducible representation of ${\bf G}$. Let $R({\bf G})$ be the radical of $\bf G$, then  ${\bf G}'= {\bf G}/R({\bf G})$ is semisimple. Since $R({\bf G})=Z({\bf G})^0$, then each element of $R({\bf G})$ acts on $V$ as a scalar by Schur's lemma. According to \cite[Theorem 10.3 and Corollary 10.4]{BT}, we know that except the trivial representation, all other irreducible representations of $\Bbbk {\bf G}'$ are infinite-dimensional. Therefore we know  that $V$ must be one-dimensional. In particular $V$ contains a $\bf B$-stable line, where $\bf B$ is a Borel subgroup of $\bf G$.
By Corollary \ref{54}, $V$ is isomorphic to $E(\theta)_\varnothing$ for some $\theta\in\widehat{\bf T}$. Since  $V$ is one-dimensional, then $I(\theta)=I$.  The corollary is proved.
\end{proof}

It is well known that any irreducible $G_q$-module over $\bar{\mathbb{F}}_q$ contains a unique $B_q$-stable line. This no longer holds for ${\bf G}$. For example, $E(\theta)_J$ has no ${\bf B}$-stable line if $J\neq\varnothing$ by the proof of Corollary \ref{54}. The classification of all abstract irreducible representations of ${\bf G}$ is still out of reach, and a new approach is needed to settle this problem.

\bigskip
\noindent{\bf Acknowledgements}\ \  The authors are grateful to
Professor Nanhua Xi for his helpful suggestions and comments in
writing this paper. The first named author would like to thank
Professor Jianpan Wang and Professor Naihong Hu for their valuable advices.  The second named author thanks Professor Toshiaki Shoji for his helpful discussion and comments. Both authors thank referees for their careful reading and valuable advices and comments to this paper.

\bibliographystyle{amsplain}

\begin{thebibliography}{10}

\bibitem {AJ}
H. H. Andersen, J. C. Jantzen, and W. Soergel, \textit{Representations of quantum groups at a $p$th root of unity and of semisimple groups in characteristic $p$: independence of $p$}, Ast$\acute{\operatorname{e}}$risque. (220): 321, 1994.

\bibitem {BT}
A Borel, J Tis, \textit{Homomorphismes "abstraits" de groupes algebriques simples}, Ann. of Math. 97(1973), 499-571.

\bibitem {Car}
R. W. Carter, \textit{Finite Groups of Lie Type: Conjugacy Classes and Complex Characters}, Pure Appl. Math. John Wiley and Sons, New York, 1985.

\bibitem {C1}
Xiaoyu Chen. \textit{On the Principal Series Representations of Semisimple Groups with Frobenius Maps.} arXiv: 1702. 05686v2.

\bibitem {C2}
Xiaoyu Chen. \textit{Some Non Quasi-finite Irreducible Representations of
Semisimple Groups with Frobenius Maps.} arxiv: 1705. 04845v1.

\bibitem {CD1}
Xiaoyu Chen, Junbin Dong. \textit{The Permutation Module on Flag Varieties in Cross Characteristic}, Math. Z. {\bf 293} (2019): 475-484.

\bibitem {CD2}
Xiaoyu Chen, Junbin Dong. \textit{The Decomposition of Permutation Module for Infinite Chevalley Groups}, Sci. China Math. to appear.

\bibitem {CE}
M Cabanes, M Enguehard, \textit{Representation theory of Finite Reductive groups}, New Math. Monogr. Cambridge University Press, Cambridge, 2004.

\bibitem {CL}
W. Carter, G. Lusztig, \textit{Modular representations of finite groups of Lie type}, Proc. Lond. Math. Soc. \textbf{32} (1976), 347--384.

\bibitem {CPSV}
E. Cline, B. Parshall, L. Scott, W. Van der Kallen, \textit{Rational and Generic Cohomology}, Invent. Math. \textbf{39} (1977), 143--163.

\bibitem {De}
D. I. Deriziotis, \textit{The Submodule Structure of Weyl Modules for Groups of Type $A_1$}, Comm. Algebra \textbf{9} (1981), 247--265.

\bibitem {DL}
P. Deligne, G. Lusztig, \textit{Representations of Reductive Groups Over Finite Fields}, Ann. of Math. \textbf{103} (1976), 103--161.

\bibitem {Fie}
P. Fiebig, \textit{Sheaves on Affine Schubert Varieties, Modular Representations, and Lusztig's conjecture}, J. Amer. Math. Soc. \textbf{24} (2011), 133--181.

\bibitem {Ge}
M. Geck, \textit{Kazhdan-Lusztig cells and the Murphy basis}, Proc. Lond. Math. Soc. \textbf{93} (2006), 635--665.

\bibitem{Haboush}
W. J. Haboush, \textit{Central Differential Operators on split semi-simple groups over fields of positive characteristic}, pp. 35--85 in: M.-P. Malliavin (ed.), {\it S$\acute{e}$minarie d'Alg$\grave{e}$bre Paul Dubriel et Marie-Paule Malliavin,} Proc. Paris 1979 (Lect. Notes Math. {\bf 795}), Berlin etc. 1980 (Springer).

\bibitem {Hum}
J. E. Humphreys, \textit{Representations of Complex Lie algebras in the BGG category $\mathcal{O}$}, Grad. Stud. Math., vol.~94, American Mathematical Society, 2008.


\bibitem {Jan1}
J. C. Jantzen, \textit{Representations of Algebraic Groups (2nd ed.)},
Math. Surveys Monogr., vol.~107, Amer. Math. Soc.,
Providence RI, 2003.

\bibitem {Jan2}
J. C. Jantzen, \textit{Filtrierungen der Darstellungen in der Hauptserie endlicher Chevalley-Gruppen}, Proc. Lond. Math. Soc. \textbf{49} (1984), 445--482.

\bibitem {KL}
D. Kazhdan, G. Lusztig, \textit{Representations of Coxeter groups and Hecke algebras}, Invent. Math. \textbf{53} (1979), 165--184.


\bibitem {L1}
G. Lusztig, \textit{Hecke algebras and Jantzen's generic decomposition patterns}, Adv. Math. 37(1980),121-164.

\bibitem {L2}
G. Lusztig, \textit{Some problems in the representation theory of finite Chevalley groups}, The Santa Cruz Conference on Finite Groups (Univ. California, Santa Cruz, Calif., 1979), Proc. Sympos. Pure Math., vol. 37, Amer. Math. Soc., Providence, R.I., 1980, pp. 313--317.

\bibitem {Pil}
C. Pillen, \textit{Loewy Series for Principal Series Representations of
Finite Chevalley Groups}, J. Algebra \textbf{189} (1997), 101--124.

\bibitem {Ru}
A. N. Rudakov, \textit{On representations of classical semisimple Lie algebras of characteristic $p$}, Math. USSR Izvestiya \textbf{4} (1970), 741--749.


\bibitem {Se}
J. P. Serre, \textit{Linear Representations of Finite Groups}, Grad. Texts in Math., vol. 42, Springer-Verlag., New York-Heidelberg, 1977.

\bibitem {Se2}
J. P. Serre, \textit{A Course in Arithmetic}, Grad. Texts in Math., vol. 7, Springer-Verlag., New York-Heidelberg, 1973.

\bibitem {Sp}
T. A. Springer, \textit{Reductive Groups}, Proc. Sympos. Pure Math., Vol. 33 (1979), part 1, pp. 3--27.

\bibitem {SW}
H. Sawada, \textit{A Characterization of the Modular Representations of Finite Groups with Split (B;N) Pairs}, Math. Z. \textbf{155} (1977), 29--41.

\bibitem {St}
R. Steinberg, \textit{Prime power representations of finite linear groups} II, Canad. J. Math. \textbf{9} (1957), 347--351.

\bibitem {Ver}
D. N. Verma, \textit{Structure of certain induced representations of complex semisimple Lie algebras}, Bull. Amer. Math. Soc. \textbf{74} (1968), 160--166; errata, 628.

\bibitem {Wil}
G. Williamson, \textit{Schubert Calculus and Torsion Explosion}, J. Amer. Math. Soc. \textbf{30} (2017), 1023--1046.

\bibitem {Xi}
Nanhua Xi, \textit{Some Infinite Dimensional Representations of Reductive Groups With Frobenius Maps}, Sci. China Math. \textbf{57} (2014), 1109--1120.

\bibitem {Yang}
Ruotao Yang, \textit{Irreducibility of Infinite Dimensional Steinberg Modules of Reductive Groups with Frobenius Maps}, J. Algebra \textbf{533} (2019), 17--24.


\bibitem {YY}
Y. Yoshida, \textit{A Generalization of Pillen's Theorem for Principal Series Modules} II, J. Algebra \textbf{429} (2015), 177--191.
\end{thebibliography}

\end{document}